\documentclass[reqno]{amsart}

\title{The Lemmens--Seidel conjecture for base size $5$}
\author[]{Kiyoto Yoshino}
\address{Department of Information Science, Faculty of Science, Toho University, 2-2-1 Miyama, Funabashi, Chiba 274-8510, Japan}
\email{kiyoto.yoshino@is.sci.toho-u.ac.jp}

\usepackage{graphicx}
\usepackage{amsmath,amsthm,amsfonts,amssymb,enumerate,mathrsfs,color,tikz}
\usepackage{multirow}
\usepackage{pgfplots}
\usetikzlibrary{calc}
\usepackage{times}
\usepackage[]{hyperref}
\usepackage{tikz,xcolor}
\usepackage {mnotes}
\usepackage[shortlabels]{enumitem}
\setlist[enumerate,1]{label={\upshape(\roman*)}}

\usepackage{geometry}
\numberwithin{equation}{section}

\newtheorem{lemma}{Lemma}[section]
\newtheorem{theorem}[lemma]{Theorem}
\newtheorem{proposition}[lemma]{Proposition}
\newtheorem{corollary}[lemma]{Corollary}
\theoremstyle{definition}
\newtheorem{definition}[lemma]{Definition}

\newtheorem{question}[lemma]{Question}
\newcommand{\R}{\mathbb{R}}
\newcommand{\Q}{\mathbb{Q}}

\newcommand{\Z}{\mathbb{Z}}

\newcommand{\cM}{\mathcal{M}}
\newcommand{\uI}{\textup{I}}
\newcommand{\uII}{\textup{II}}
\newcommand{\bu}{\mathbf{u}}
\newcommand{\bv}{\mathbf{v}}
\newcommand{\bx}{\mathbf{x}}

\newcommand{\bw}{\mathbf{w}}
\newcommand{\bff}{\mathbf{f}}

\begin{document}


\tikzset{
			e/.style={line width=1 pt,
				line join=round,
				line cap=round},
			slim/.style={
				circle,
				fill=black,
				draw=black,
				minimum size=5 pt,
				inner sep=0pt},
			wslim/.style={
				circle,
				draw=black,
				fill=white,
				minimum size=5 pt,
				inner sep=0pt
				},
			lodot/.style={loosely dotted,
				line width=2 pt	,
				line join=round,
				line cap=round}
}
\newcommand{\Dp}{
	\begin{tikzpicture}[		
		baseline=0pt,
		xscale=0.5,
		yscale=0.5]
		\draw [e] (0,0) node[slim, label=below:$1$] () {}
			--(1,0) node[slim, label=below:$2$] (b) {}
			--(2,0) node[slim, label=below:$3$] (c) {}
			--(3,0) node[slim, label=below:$4$] () {}
			--(4,0) node[slim, label=below:$5$] (f) {}
			(6,0) node[slim, label=below:$t-2$] (g) {}
			--(7,0) node[slim, label=right:$t-1$] () {}
			(b.center)--(1,1) node[slim, label=right:$t$] () {};
		\draw [lodot] (f.center)--(g.center);
	\end{tikzpicture}
}
\newcommand{\Ap}{
	\begin{tikzpicture}[		
		baseline=0pt,
		xscale=0.5,
		yscale=0.5]
		\draw [e] (0,0) node[slim] () {}
			--(1,0) node[slim] (b) {}
			--(2,0) node[wslim] (c) {}
			--(3,0) node[slim] () {}
			--(4,0) node[slim] (f) {}
			--(5,0) node[wslim] (g) {}
			--(6,0) node[slim] (h) {}
			(8,0) node[white] (i) {};
		\draw [lodot] (h.center)--(i.center);
	\end{tikzpicture}
}
\newcommand{\tDp}{
	\begin{tikzpicture}[		
		baseline=0pt,
		xscale=0.5,
		yscale=0.5]
		\draw [e] (0,0) node[slim] () {}
			--(1,0) node[wslim] (b) {}
			--(2,0) node[slim] (c) {}
			--(3,0) node[slim] () {}
			--(4,0) node[wslim] () {}
			--(5,0) node[slim] (f) {}
			(7,0) node[slim] (g) {}
			--(8,0) node[wslim] (h) {}
			--(9,0) node[slim] () {}
			(b.center)--(1,1) node[slim] () {}
			(h.center)--(8,1) node[slim] () {};
		\draw [lodot] (f)--(g);
	\end{tikzpicture}
}
\newcommand{\tAp}{
	\begin{tikzpicture}[		
		baseline=0pt,
		xscale=0.5,
		yscale=0.5]
		\draw [e] 
			(0,0) node[wslim] (a) {}
			--(1,-1) node[slim] (b) {}
			--(2,-1) node[slim] (c) {}
			--(3,-1) node[wslim] (d) {}
			--(4,-1) node[slim] (e) {}
			(6,-1) node[wslim] (f) {}
			--(7.2,-0.5) node[slim] () {}
			--(7.2,0.5) node[slim] () {}
			--(6,1) node[wslim] (h) {}
			(4,1) node[slim] (i) {}
			--(3,1) node[wslim] () {}
			--(2,1) node[slim] () {}
			--(1,1) node[slim] () {}
			--(a.center);
		\draw [lodot] (e)--(f) (h)--(i);
	\end{tikzpicture}
}
\newcommand{\Ef}{
	\begin{tikzpicture}[		
		baseline=0pt,
		xscale=0.5,
		yscale=0.5]
		\draw [e] 
			(0,0) node[slim] () {}
			--(1,0) node[slim] () {}
			--(2,0) node[wslim] (a) {}
			--(3,0) node[slim] () {}
			--(4,0) node[slim] () {}
			(a.center)
			--(2,1) node[slim] () {};		
	\end{tikzpicture}
}
\newcommand{\tEf}{
	\begin{tikzpicture}[		
		baseline=0pt,
		xscale=0.5,
		yscale=0.5]
		\draw [e] 
			(0,0) node[slim] () {}
			--(1,0) node[slim] () {}
			--(2,0) node[wslim] (a) {}
			--(3,0) node[slim] () {}
			--(4,0) node[slim] () {}
			(a.center)
			--(2,1) node[slim] () {}		
			--(2,2) node[slim] () {};
	\end{tikzpicture}
}
\newcommand{\Eg}{
	\begin{tikzpicture}[		
		baseline=0pt,
		xscale=0.5,
		yscale=0.5]
		\draw [e] 
			(-1,0) node[slim] () {}		
			--(0,0) node[slim] () {}
			--(1,0) node[wslim] () {}
			--(2,0) node[wslim] (a) {}
			--(3,0) node[slim] () {}
			--(4,0) node[slim] () {}
			(a.center)
			--(2,1) node[slim] () {};		
	\end{tikzpicture}
}
\newcommand{\tEg}{
	\begin{tikzpicture}[		
		baseline=0pt,
		xscale=0.5,
		yscale=0.5]
		\draw [e] 
			(-1,0) node[slim] () {}		
			--(0,0) node[slim] () {}
			--(1,0) node[wslim] () {}
			--(2,0) node[slim] (a) {}
			--(3,0) node[wslim] () {}
			--(4,0) node[slim] () {}
			--(5,0) node[slim] () {}
			(a.center)
			--(2,1) node[slim] () {};		
	\end{tikzpicture}
}
\newcommand{\Eh}{
	\begin{tikzpicture}[		
		baseline=0pt,
		xscale=0.5,
		yscale=0.5]
		\draw [e] 
			(0,0) node[slim] () {}
			--(1,0) node[slim] () {}
			--(2,0) node[wslim] (a) {}
			--(3,0) node[slim] () {}
			--(4,0) node[slim] () {}
			--(5,0) node[wslim] () {}
			--(6,0) node[slim] () {}
			(a.center)
			--(2,1) node[slim] () {};		
	\end{tikzpicture}
}
\newcommand{\tEh}{
	\begin{tikzpicture}[		
		baseline=0pt,
		xscale=0.5,
		yscale=0.5]
		\draw [e] 
			(0,0) node[slim] () {}
			--(1,0) node[slim] () {}
			--(2,0) node[wslim] (a) {}
			--(3,0) node[slim] () {}
			--(4,0) node[slim] () {}
			--(5,0) node[wslim] () {}
			--(6,0) node[slim] () {}
			--(7,0) node[slim] () {}
			(a.center)
			--(2,1) node[slim] () {};		
	\end{tikzpicture}
}
\maketitle

\begin{abstract}
	In 2020, Lin and Yu claimed to prove the so-called Lemmens--Seidel conjecture for base size $5$.
	However, their proof has a gap.
	In this paper, we prove the conjecture for base size $5$ using the pillar method.
	We also show that the sets of $57$ equiangular lines with common angle $\arccos(1/5)$ in dimension $18$ found by Greaves et~al.\ in 2021 are indeed counterexamples to one of Lin and Yu's claims.
	We prove this by answering the question posed by Greaves et~al.\ in 2021 in the negative.
	They asked whether these sets are contained in the unique set of $276$ equiangular lines with common angle $\arccos(1/5)$ in dimension $23$.
	Furthermore, we show that these sets are strongly maximal.
	This gives a negative answer to the question posed by Cao et~al.\ in 2021.
	They asked whether the unique set of $276$ equiangular lines with common angle $\arccos(1/5)$ in dimension $23$ is the unique strongly maximal set of equiangular lines with common angle $\arccos(1/5)$.
\end{abstract}

\section{Introduction}
A set of lines through the origin in a Euclidean space is \emph{equiangular} if any pair from these lines forms the same angle.
The problem of determining the maximum cardinality of a set of equiangular lines in a Euclidean space dates back to the result of Haantjes~\cite{Haantjes1948}.
Denote by $N(d)$ the maximum cardinality of a set of equiangular lines in dimension $d$.
The values of $N(d)$ are known for $d \leq 43$ except for $d = 18, 19, 20, 42$~\cite{Barg2017, Greaves2020, Greaves2021, lemmens1973, Lint1966}.
Also, Gerzon proved the so-called absolute bound $N(d) \leq d(d+1)/2$~\cite[Theorem~3.5]{lemmens1973}.
If equality holds, then $d+2$ is $4$, $5$ or the square of an odd integer at least 3.
On the other hand, a lower bound was given in \cite{de2000}, and was improved in \cite[Corollary~2.8]{GHAF2016} to $N(d) \geq (32d^2+328d+296)/1089$ for $d \geq 2$.

For a fixed angle, sets of equiangular lines have been studied.
Denote by $N_\alpha(d)$ the maximum cardinality of a set of equiangular lines with common angle $\arccos(\alpha)$ in dimension $d$.
Lemmens and Seidel proved that for a set of $n$ equiangular lines with common angle $\arccos(\alpha)$ in dimension $d$,
$1/\alpha$ is an odd integer if $n > 2d$~\cite[Theorem~3.4]{lemmens1973}.
In low dimensions, they proved $N_{\alpha}(d) \leq d(1-\alpha^2)/(1-d\alpha^2)$ for $d < 1/\alpha^2$.
In high dimensions, Jiang et~al.\ proved for every integer $k \geq 2$, $N_{1/(2k-1)}(d) = \lfloor k(d-1)/(k-1) \rfloor$ for all sufficiently large $d$~\cite[Corollary~1.3]{Zhao2021}.

In the case where the common angle is $\arccos(1/3)$ or $\arccos(1/5)$, sets of equiangular lines have been investigated in detail.
For the common angle $\arccos(1/3)$, Lemmens and Seidel introduced the pillar method, and determined the values of $N_{1/3}(d)$ for all $d$~\cite[Theorem~3.6]{lemmens1973}.
This method is the main tool to prove our result, and will be given in Definition~\ref{dfn:pillar}.
Following a different approach, Cao et~al.~\cite{cao2021} and Yoshino~\cite{Y2025} investigated sets of equiangular lines with common angle $\arccos(1/3)$ more precisely using root lattices.
For the common angle $\arccos(1/5)$, the values and bounds of $N_{1/5}(d)$ have been investigated for $d < 23$~\cite{Greaves2022, Greaves2020, Greaves2021, lemmens1973, Lin2020b, Ferenc2019, Daniel1971}.
Among these, the values for $d=18,19,20$ remain open.
For $d \geq 23$, Lemmens and Seidel raised the following, which is the so-called Lemmens--Seidel conjecture.
\begin{theorem}[The Lemmens--Seidel conjecture]	\label{thm:Lemmens--Seidel}
	For $d \geq 23$,
	$
		N_{1/5}(d) = \max\left\{
		276,
		\left\lfloor (3d-3)/2 \right\rfloor
		\right\}.
	$
\end{theorem}
Here, a set of $276$ equiangular lines with common angle $\arccos(1/5)$ in dimension $23$ is known to be unique~\cite[Theorem~A]{goethals1975}.
In order to prove the Lemmens--Seidel conjecture, we need to show it for base sizes $3$,$4$,$5$ and $6$,
where the base size will be given in Definition~\ref{dfn:base size}.
In 1973, Lemmens and Seidel proved it for base size $6$~\cite[Theorem~5.7]{lemmens1973}.
In 2020, Lin and Yu proved the conjecture for base size $3$ with a computer~\cite[Theorem~4.3]{Lin2020},
and claimed to prove it for base size $5$~\cite[Theorem~4.6]{Lin2020}.
However, there is a gap in~\cite[Proof of Theorem~4.6~(1)]{Lin2020}.
This gap undermines their main approach, and therefore a different approach is needed.
In 2022, Cao et~al.\ proved the conjecture for base sizes $3$ and $4$ without a computer~\cite[Theorems~6.1, 7.5 and 9.3]{CKLY2022}.
Hence, to complete a proof of the conjecture, we prove Theorem~\ref{thm:base=5} by a different approach.
This immediately implies the conjecture for base size $5$.
\begin{theorem}[The Lemmens--Seidel conjecture for base size $5$]	\label{thm:base=5}
	A set of $n$ equiangular lines with common angle $\arccos(1/5)$ with base size $5$ in dimension $d$ satisfies
	$
		n \leq \max\left\{
		276,
		\left\lfloor (4d +36)/3 \right\rfloor
		\right\}.
	$
\end{theorem}
Our proof of Theorem~\ref{thm:base=5} uses the pillar method of Lemmens and Seidel.
The pillar method gives a partition of a set of equiangular lines into a base and pillars.
A key point in the argument for base size $6$ is that an upper bound on the cardinality of a pillar can be obtained if there exists another non-empty pillar.
For base size $5$, to obtain the required upper bound for a specific pillar, we need to consider two or three other pillars simultaneously.
The choice of these pillars is crucial.
In Corollary~\ref{cor:ga}, we specify such a choice and prove the corresponding upper bound.
We remark that, in this paper, the pillars are defined in terms of graphs rather than in terms of lines.

After completing the proof of the conjecture, we describe a counterexample to a claim of Lin and Yu in~\cite[Proof of Theorem~4.6~(1)]{Lin2020}.
The gap in~\cite[Proof of Theorem~4.6~(1)]{Lin2020} is in claiming that a set of equiangular lines with common angle $\arccos(1/5)$, base size $6$ and at least two pillars containing edges must lie in a unique set of $276$ equiangular lines in dimension $23$.
Since a set of equiangular lines with base size $5$ without $(5,1)$-pillars can be regarded as one with base size $6$ by adding some extra line,
they discussed sets of equiangular lines with base size $6$.
However, the four sets of $57$ equiangular lines with common angle $\arccos(1/5)$ in dimension $18$ induced by the Seidel matrices written as $S_1, S_2, S_3$ and $S_4$ in~\cite{Greaves2021} are counterexamples to their claim.
To show it, we answer the following question in the negative in Proposition~\ref{prop:Gary}, and verify a straightforward fact in Proposition~\ref{prop:counter}.
\begin{question}[{\cite[Question~2.1]{Greaves2021}}]	\label{ques:Gary}
	Can the two sets of $57$ equiangular lines in dimension $18$ corresponding to the Seidel matrices $S_1$ and $S_2$ in~\cite{Greaves2021} be found inside the set of $276$ equiangular lines in dimension $23$?
\end{question}
Although Greaves et~al.\ found more sets of $57$ equiangular lines in dimension $18$~\cite{Greaves2021}, they posed a question about only two Seidel matrices $S_1$ and $S_2$.
Note that we treat the four Seidel matrices $S_1, S_2, S_3$ and $S_4$, which are explicitly given in~\cite{Greaves2021}.

Using these examples, we may further explain that the situations for the common angles $\arccos(1/3)$ and $\arccos(1/5)$ are substantially different.
It is well known that for each $1/\alpha \in \{2,\sqrt{5},3,5\}$ there exists a set of equiangular lines with common angle $\arccos(\alpha)$ attaining the absolute bound, namely $3$, $6$, $28$, and $276$ in dimensions $2$, $3$, $7$, and $23$, respectively.
In the case of the common angle $\arccos(1/3)$, every set of equiangular lines in dimension at most $7$ lies in the unique set of $28$ equiangular lines in dimension $7$~\cite{Y2025}, which achieves the absolute bound.
Thus, Lin and Yu's claim was plausible because it relied on an analogy with the case $\arccos(1/3)$.
However, we see that such an analogy does not hold for $\arccos(1/5)$ by the counterexamples.

In light of this, the extension problem for equiangular lines is of independent interest.
Cao et~al.\ introduced the concept of strong maximality and showed that a set of equiangular lines is strongly maximal if it achieves the absolute bound~\cite[Theorem~5.5]{cao2021}.
Here, a set $U$ of equiangular lines is said to be \emph{strongly maximal} if there is no set of equiangular lines properly containing $U$ even in higher dimensions.
Furthermore, Cao et~al.\ proved uniqueness of strongly maximal sets of equiangular lines with common angle $\arccos(\alpha)$ for $1/\alpha \in \{2,\sqrt{5},3 \}$, and posed Question~\ref{ques:Cao}.
\begin{question}[{\cite[Question~5.7]{cao2021}}]	\label{ques:Cao}
	Is the set of $276$ equiangular lines with common angle $\arccos(1/5)$ in dimension $23$ the unique strongly maximal set of equiangular lines with common angle $\arccos(1/5)$?
\end{question}
We show that the sets of equiangular lines induced by the Seidel matrices $S_1, S_2, S_3$ and $S_4$ are strongly maximal in Proposition~\ref{prop:strongly}.
As a result, we answer the question in the negative.

This paper is organized as follows.
In Section~\ref{sec:notations}, we introduce some concepts in connection with equiangular lines, and explain some notations.
In Section~\ref{sec:Lemmens--Seidel}, we rewrite Lemmens and Seidel's results for sets of equiangular lines with common angle $\arccos(1/5)$ and base size $5$.
In Section~\ref{sec:(5,2)-pillar mK2}, we prove a key theorem, Theorem~\ref{thm:key},
to show the Lemmens--Seidel conjecture for base size $5$.
In Section~\ref{sec:(5,2)-pillar}, we give an upper bound on the order of a $(5,2)$-pillar under some assumptions.
In Section~\ref{sec:proof}, we prove the Lemmens--Seidel conjecture for base size $5$.
In Section~\ref{sec:ques}, we show some properties of sets of $57$ equiangular lines with common angle $\arccos(1/5)$ in dimension $18$ found by Greaves et~al.~\cite{Greaves2021}, and answer Questions~\ref{ques:Gary} and~\ref{ques:Cao} in the negative.

\section{Notations}	\label{sec:notations}
Throughout this paper, we will consider undirected graphs, without loops or multiple edges.
Let $H$ be a graph.
Denote by $V(H)$ the set of vertices, and by $E(H)$ the set of edges.
Write $x \sim y$ for $x,y \in V(H)$ if $x$ and $y$ are adjacent in $H$.
Denote by $N(x)=N_H(x)$ the set of neighbors of a vertex $x$ in $H$.
We write $G+H$ for the disjoint union of two graphs $G$ and $H$.
For a non-negative integer $m$, we write $mH$ for the disjoint union of $m$ copies of $H$.
A \emph{clique} in $H$ is an induced subgraph isomorphic to a complete graph,
and a \emph{maximum} clique is a clique such that there is no clique with more vertices.
We will identify a clique with its vertex set.
The \emph{clique number} of $H$ is the order of the largest clique in $H$.
For a subset $W$ of $V(H)$, denote by $H^W$ the graph with vertex set $V(H)$ and edge set $E(H) \ \triangle \ \{ \{x,y\} : x \in W, y \in V(H) \setminus W \}$, where $\triangle$ denotes symmetric difference.
We say that $H^W$ is the graph obtained from $H$ by switching with respect to $W$.
Two graphs $G$ and $H$ are said to be \emph{switching equivalent} if there exists a subset $W \subset V(H)$ such that $G$ and $H^W$ are isomorphic.

Denote by $I$ and $J$ the identity matrix and the all-ones matrix, respectively.
If the size of each matrix is not clear, then we will indicate its size by a subscript.

A \emph{Seidel matrix} is a symmetric matrix with zero diagonal and all off-diagonal entries $\pm1$.
Two Seidel matrices $S$ and $S'$ are said to be \emph{switching equivalent} if there exist a permutation matrix $P$ and a diagonal matrix $D$ with diagonal entries $\pm 1$ such that $(PD)^\top S (PD) = S'$ holds.
For a graph $H$, denote by $A(H)$ the adjacency matrix of $H$, and define $S(H) := J-I-2A(H)$.
Note that for any Seidel matrix $S$, there exists a graph $H$ such that $S=S(H)$.
The eigenvalues of $S(H)$ are called the \emph{Seidel eigenvalues} of $H$.
Note that two graphs $G$ and $H$ are switching equivalent if and only if $S(G)$ and $S(H)$ are switching equivalent.

Fix a set $U$ of $n$ equiangular lines with common angle $\arccos(\alpha)$.
Then we may take unit vectors $\bu_1,\ldots,\bu_n$ such that $U = \{ \R \bu_1, \ldots, \R \bu_n \}$.
There exists a Seidel matrix $S$ such that
$
	I +  \alpha S
$
equals the Gram matrix of $\bu_1,\ldots,\bu_n$.
Hence the set $U$ of equiangular lines induces the Seidel matrix $S$ up to switching.
In addition, it induces the graph $H$ with $S=S(H)$ up to switching.
Note that the smallest Seidel eigenvalue of $H$ is at least $-1/\alpha$.
Conversely, we can recover $U$ from the Seidel matrix $S$ or the graph $H$.

The pillar method introduced by Lemmens and Seidel~\cite{lemmens1973} has been employed to investigate equiangular lines~\cite{CKLY2022,KT2019, lemmens1973, Lin2020}.
Since this paper also employs this method, we define some notations and explain the method here.
\begin{definition}	\label{dfn:base size}
	Let $U$ be a set of equiangular lines, and $H$ a graph induced by $U$.
	The maximum of the clique numbers of graphs switching equivalent to $H$ is called the \emph{base size} of $U$.
\end{definition}

Denote by $\langle \bv_1,\ldots,\bv_n \rangle$ the linear space generated by vectors $\bv_1,\ldots,\bv_n$.
Also denote by $(\bu,\bv)$ the inner product of two vectors $\bu$ and $\bv$, and call $(\bu,\bu)$ the norm of a vector $\bu$.

\begin{definition}	\label{dfn:pillar}
	Let $H$ be a graph, and $B$ be a maximum clique.
	A \emph{pillar} $P_{B,T}$ with respect to $B$ for a subset $T \subset B$ is defined to be the induced subgraph in $H$ on
	\begin{align*}
		\{ x \in V(H) \setminus B \mid N(x) \cap B = T \}.
	\end{align*}
	Moreover, this is called a \emph{$(|B|,|T|)$-pillar}.
	Also, assume that the smallest Seidel eigenvalue of $H$ is at least $-5$.
	Then denote by $\hat{x}$'s the vectors such that
	\begin{align*}
		( \hat{x}, \hat{y} ) = (I+S(H)/5)_{xy} \qquad (x,y \in V(H)).
	\end{align*}
	Denote by $\bar{x}$ the orthogonal projection of $\hat{x}$ onto the subspace $\langle \hat{b} : b \in B \rangle^\perp$.
\end{definition}

We explain the pillar method in terms of graphs.
Fix a set of equiangular lines with common angle $\arccos(\alpha)$ and base size $K$.
There exists a graph $H$ induced by the set of equiangular lines with smallest Seidel eigenvalue at least $-1/\alpha$ such that $H$ has a maximum clique $B$ of order $K$.
Then, since $K$ is the clique number of $H$, we see that $H$ has no $(K,K)$-pillar with respect to $B$.
Furthermore, if a pillar $P$ is a $(K,i)$-pillar with respect to $B$ in $H$, then $P$ is a $(K,K-i)$-pillar with respect to $B$ in $H^{V(P)}$.
Hence, we may assume any pillar with respect to $B$ in $H$ containing at least one vertex is a $(K,i)$-pillar for some $i \in \{1,\ldots,\lfloor K/2 \rfloor \}$.
In the pillar method, $V(H) \setminus B$ is partitioned into $(K,i)$-pillars, where $1 \leq i \leq \lfloor K/2 \rfloor$.

Subsequently, an upper bound on the order of each pillar is often given by considering a semidefinite programming problem~\cite{CKLY2022,KT2019, lemmens1973, Lin2020}.
In this paper, we will deal with the case of $\alpha = 1/5$.
Hence, we will consider $(5,1)$-pillars and $(5,2)$-pillars,
and provide upper bounds on the order of a pillar under the condition that the Gram matrix of vectors $\bar{x}$'s ($x \in V(H) \setminus B$) is positive semidefinite.

To obtain such an upper bound, we will use the classification of graphs with largest eigenvalue at most $2$.
The connected graphs with largest eigenvalue at most $2$ are enumerated in Figure~\ref{fig:CD} (cf.~\cite[Theorem~3.1.3]{brouwer2011spectra}).
\begin{figure}[hbtp]
	\begin{align*}
		 & D_{t}	:=	\Dp \quad (t \geq 4)	\quad
		 &                                     & \tilde{D}_{t}	:=	\tDp \quad (t \geq 4) \\
		 & A_{t}	:=	\Ap \quad (t \geq 1)	\quad
		 &                                     & \tilde{A}_{t}	:=	\tAp \quad (t \geq 2) \\
		 & E_6 		:=	\Ef 		\quad
		 &                                     & \tilde{E}_6 	:=	\tEf                   \\
		 & E_7 		:=	\Eg		\quad
		 &                                     & \tilde{E}_7 	:=	\tEg                   \\
		 & E_8 		:=	\Eh		\quad
		 &                                     & \tilde{E}_8 	:=	\tEh
	\end{align*}
	\caption{The connected graphs having largest eigenvalue at most $2$, where the colors of vertices will be used in Lemma~\ref{lem:dynkin} } \label{fig:CD}
\end{figure}

\section{Lemmens and Seidel's results}	\label{sec:Lemmens--Seidel}
Lemmens and Seidel introduced the pillar method, and proved the Lemmens--Seidel conjecture for base size $6$.
Then, they investigated pillars in graphs with smallest Seidel eigenvalue at least $-5$ and clique number $6$.
In this section, we obtain some results on $(5,2)$-pillars from Lemmens and Seidel's results on $(6,3)$-pillars in~\cite{lemmens1973}.
To demonstrate, we write one of Lemmens and Seidel's results in the notation of this paper as follows.
\begin{theorem}[{\cite[Theorem~5.2]{lemmens1973}}]	\label{thm:lemmens1973-theorem5.2}
	Let $G$ be a graph with smallest Seidel eigenvalue at least $-5$ and clique number $6$.
	Then any pillar contains at most $27$ vertices if another pillar contains an edge.
\end{theorem}
They remarked that in a graph with smallest Seidel eigenvalue at least $-5$ and clique number $6$, the only pillars containing vertices with respect to a maximum clique of order $6$ are $(6,3)$-pillars~\cite[p.~502]{lemmens1973}.
Indeed, we fix such a clique $B=\{b_1,\ldots,b_6\}$ and a vertex $x \not\in B$.
Then, we see that $\hat{b}_1, \ldots, \hat{b}_6$ form a $5$-simplex and $\sum_{i=1}^6  \hat{b}_i = 0$.
Hence, half of $(\hat{x} , \hat{b}_1), \ldots, (\hat{x} , \hat{b}_6)$ are $1/5$, and the remaining half are $-1/5$.
This means that $|N(x) \cap B| = 3$, and thus $x$ is contained in some $(6,3)$-pillar.
Therefore, the pillars in Theorem~\ref{thm:lemmens1973-theorem5.2} are $(6,3)$-pillars.

Also, as discussed in~\cite[Proof of Theorem~4.6]{Lin2020}, $(5,2)$-pillars can be regarded as $(6,3)$-pillars by adding an extra vertex.
We restate the discussion as the following lemma, and provide a proof for the convenience of the readers.
\begin{lemma} \label{lem:56}
	Let $H$ be a graph with smallest Seidel eigenvalue at least $-5$ having a maximum clique $B=\{b_1,\ldots,b_5\}$.
	Assume that the only pillars in $H$ with respect to $B$ containing at least one vertex are $(5,2)$-pillars.
	Then there exists a supergraph $G$ of $H$ with smallest Seidel eigenvalue at least $-5$ and a vertex $b_6 \in V(G)$ satisfying the following.
	\begin{enumerate}
		\item $V(H) \cup \{b_6\} = V(G)$, and $B \cup \{b_6\}$ is a maximum clique.	\label{lem:56:1}
		\item The $(5,2)$-pillars in $H$ with respect to $B$ coincide with the $(6,3)$-pillars adjacent to $b_6$ in $G$ with respect to $B \cup \{b_6\}$.	\label{lem:56:2}
	\end{enumerate}
\end{lemma}
\begin{proof}
	We define the desired graph $G$ by adding an extra vertex $b_6$ to $H$ such that $N_G(b_6) = V(H)$.
	Then,~\ref{lem:56:1} and~\ref{lem:56:2} hold.
	Hence, it suffices to show that the smallest Seidel eigenvalue of $G$ is at least $-5$.
	The vectors $\hat{x}$'s given in Definition~\ref{dfn:pillar} satisfy $( \hat{x}, \hat{y} ) = (I+S(H)/5)_{xy}$ for $x,y \in V(H)$.
	We let $\hat{b}_6 := -\hat{b}_1-\hat{b}_2-\hat{b}_3-\hat{b}_4-\hat{b}_5$.
	Then, $( \hat{x}, \hat{y} ) = (I+S(G)/5)_{xy}$ for $x,y \in V(G)$.
	Therefore, the smallest Seidel eigenvalue of $G$ is at least $-5$.
\end{proof}
By this lemma, we may rewrite some results in~\cite{lemmens1973}.
In fact, we obtain the following three theorems by combining Lemma~\ref{lem:56} with Lemmens and Seidel's results.
We provide a proof of only Theorem~\ref{thm:|P|}~\ref{thm:|P|:1} at the end of this section.
We may prove the other theorems in the same way.

\begin{theorem}[{\cite[Theorem~5.2]{lemmens1973}}] \label{thm:5.2}
	Let $H$ be a graph with smallest Seidel eigenvalue at least $-5$ and clique number $5$.
	Let $P$ be a $(5,2)$-pillar.
	If another $(5,2)$-pillar contains at least one vertex, then the following hold.
	\begin{enumerate}
		\item If an induced subgraph of $P$ is isomorphic to $\tilde{A}_t$, then $t+1 \bmod 3 = 0$.
		\item If an induced subgraph of $P$ is isomorphic to $\tilde{D}_t$, then $t+1 \bmod 3 = 2$.
	\end{enumerate}
\end{theorem}

\begin{theorem}[{\cite[Theorems~5.3, 5.4 and 5.5]{lemmens1973}}]	\label{thm:2m+n}
	Let $H$ be a graph with smallest Seidel eigenvalue at least $-5$ and clique number $5$.
	Assume a $(5,2)$-pillar is isomorphic to $m K_2 + n K_1$ for some non-negative integers $m$ and $n$.
	\begin{enumerate}
		\item $2m+n \leq 18$ if another $(5,2)$-pillar contains an edge.
		\item $2m+n \leq 24$ if another $(5,2)$-pillar contains non-adjacent vertices.
		\item $2m+n \leq 36$ if another $(5,2)$-pillar contains a vertex.
	\end{enumerate}
\end{theorem}

\begin{theorem}[{\cite[Proof of Theorem~5.6 with Theorems~5.3, 5.4 and 5.5]{lemmens1973}}]	\label{thm:|P|}
	Let $H$ be a graph with smallest Seidel eigenvalue at least $-5$ and clique number $5$.
	Let $P$ be a $(5,2)$-pillar.
	\begin{enumerate}
		\item $|V(P)| \leq 27$ if another $(5,2)$-pillar contains an edge.	\label{thm:|P|:1}
		\item $|V(P)| \leq 36$ if another $(5,2)$-pillar contains non-adjacent vertices.
		\item $|V(P)| \leq 54$ if another $(5,2)$-pillar contains a vertex.
	\end{enumerate}
\end{theorem}
\begin{proof}[{Proof of Theorem~\ref{thm:|P|}~\ref{thm:|P|:1}}]
	Fix a maximum clique $B = \{b_1,\ldots,b_5\}$ of $H$.
	Assume a pillar $P'$ of $H$ with respect to $B$ distinct from $P$ is a $(5,2)$-pillar containing an edge, and prove $|V(P)| \leq 27$.
	Let $H'$ be the graph obtained from $H$ by removing all the vertices contained in $(5,1)$-pillars with respect to $B$.
	Then, $P$ and $P'$ are $(5,2)$-pillars of $H'$ with respect to $B$.
	By Lemma~\ref{lem:56}, there exists a supergraph $G$ of $H'$ with smallest Seidel eigenvalue at least $-5$ and a vertex $b_6 \in V(G)$ satisfying~\ref{lem:56:1} and~\ref{lem:56:2} in Lemma~\ref{lem:56}.
	In particular, the clique number of $G$ is $6$.
	In addition, $P$ and $P'$ are distinct $(6,3)$-pillars in $G$ with respect to $B \cup \{ b_6 \}$.
	Since $P'$ contains an edge, $|V(P)| \leq 27$ follows from Theorem~\ref{thm:lemmens1973-theorem5.2}.
\end{proof}

\section{$(5,2)$-pillars isomorphic to $mK_2$ and $(5,1)$-pillars}	\label{sec:(5,2)-pillar mK2}
The known upper bounds on the order of each of $(5,1)$- and $(5,2)$-pillars of a set of equiangular lines with common angle $\arccos(1/5)$ are useful.
However, smaller upper bounds are required to prove the Lemmens--Seidel conjecture.
The known upper bounds were provided under the condition that $(5,1)$-pillars containing vertices and $(5,2)$-pillars containing vertices do not coexist.
Thus, we will provide a smaller upper bound on the order of a $(5,2)$-pillar by considering the case where they coexist as in Figure~\ref{fig:key}.
The desired upper bound will be given as Corollary~\ref{cor:ga} in Section~\ref{sec:(5,2)-pillar}.
In this section, we establish the following theorem to prove it.
\begin{figure}[htbp]
	$\textrm{(I)} \qquad$
	\begin{tikzpicture}[baseline=7,xscale=0.6,yscale=0.6,every node/.style={fill=black,circle,minimum size=5 pt,inner sep=0pt}]
		\def\R{1}
		\node[] (b1) at (0:\R) {};
		\node[] (b2) at (72:\R) {};
		\node[] (b3) at ($(2*72:\R)$) {};
		\node[] (b4) at ($(3*72:\R)$) {};
		\node[] (b5) at ($(4*72:\R)$) {};

		\def\r{1.6}
		{
			\node[fill=white, rectangle] () at (-12:\r) {$b_1$};
			\node[fill=white, rectangle] () at (89:\r) {$b_2$};
			\node[fill=white, rectangle] () at (149:\r) {$b_3$};
			\node[fill=white, rectangle] () at (231:\r) {$b_4$};
			\node[fill=white, rectangle] () at (303:\r) {$b_5$};
		}

		\foreach \i in {b1,b2,b3,b4,b5}{
				\foreach \j in {b1,b2,b3,b4,b5}{
						\draw (\i)--(\j);
					}
			}
		\draw ($(0:\R)$)--($(0:3)$);
		\draw[fill=white] ($(0:3)$) ellipse [x radius=1, y radius=1];
		\draw ($(0:3)+(0,-1.5)$) node[fill=white, rectangle] { $P_{B,\{b_1 \}}$};

		\draw ($(144:\R)$)--($(180:3)$);
		\draw ($(216:\R)$)--($(180:3)$);
		\draw[fill=white] ($(180:3)$) ellipse [x radius=1.3, y radius=1.3];
		\draw ($(180:3)+(0,-1.8)$) node[fill=white, rectangle] { $P_{B,\{b_3, b_4 \}}$};

		\draw ($(0:\R)$)--($(50:5)$);
		\draw ($(72:\R)$)--($(50:5)$);
		\draw[fill=white] ($(55:5)$) ellipse [x radius=2.5, y radius=2.5];
		\draw ($(55:5)+(0,3)$) node[fill=white, rectangle] { $P_{B,\{b_1, b_2 \}}$};

		\node[label = above: $z_{2}$] () at ($(0*72:3)+(0.4,0)$) {};
		\node[label = above: $z_{1}$] () at ($(0*72:3)+(-0.4,0)$) {};
		\node[label = above: $y_{2}$] (c1) at ($(180:3)+(0.4,0)$) {};
		\node[label = above: $y_{1}$] (c2) at ($(180:3)+(-0.4,0)$) {};
		\draw (c1)--(c2);

		\node[label = right: $x_{1,2}$] (a1) at ($(55:5)+(0.4,1)$) {};
		\node[label = left: $x_{1,1}$] (a2) at ($(55:5)+(-0.4,1)$) {};
		\node[label = right: $x_{2,2}$] (a3) at ($(55:5)+(0.4,0.2)$) {};
		\node[label = left: $x_{2,1}$] (a4) at ($(55:5)+(-0.4,0.2)$) {};
		\draw ($(55:5)+(0,-0.2)$) node[fill=white, rectangle] { $\vdots$};
		\node[label = right: $x_{m,2}$] (a5) at ($(55:5)+(0.4,-1)$) {};
		\node[label = left: $x_{m,1}$] (a6) at ($(55:5)+(-0.4,-1)$) {};
		\draw (a1)--(a2) (a3)--(a4) (a5)--(a6);
	\end{tikzpicture}
	\hspace{15pt}
	$\textrm{(II)}\qquad$
	\begin{tikzpicture}[baseline=7,xscale=0.6,yscale=0.6,every node/.style={fill=black,circle,minimum size=5 pt,inner sep=0pt}]
		\def\R{1}
		\node[] (b1) at (0:\R) {};
		\node[] (b2) at (72:\R) {};
		\node[] (b3) at ($(2*72:\R)$) {};
		\node[] (b4) at ($(3*72:\R)$) {};
		\node[] (b5) at ($(4*72:\R)$) {};

		\def\r{1.6}
		{\
			\node[fill=white, rectangle] () at (3:\r) {$b_1$};
			\node[fill=white, rectangle] () at (93:\r) {$b_2$};
			\node[fill=white, rectangle] () at (149:\r) {$b_3$};
			\node[fill=white, rectangle] () at (231:\r) {$b_4$};
			\node[fill=white, rectangle] () at (303:\r) {$b_5$};
		}

		\foreach \i in {b1,b2,b3,b4,b5}{
				\foreach \j in {b1,b2,b3,b4,b5}{
						\draw (\i)--(\j);
					}
			}
		\draw ($(0:\R)$)--($(-20:3)$);
		\draw[fill=white] ($(-20:3)$) ellipse [x radius=1, y radius=1];
		\draw ($(-20:3)+(0,-1.5)$) node[fill=white, rectangle] { $P_{B,\{b_1\}}$};

		\draw ($(72:\R)$)--($(90:3)$);
		\draw[fill=white] ($(90:3)$) ellipse [x radius=1, y radius=1];
		\draw ($(90:3)+(0,1.5)$) node[fill=white, rectangle] { $P_{B,\{b_2\}}$};
		\draw ($(0:\R)$)--($(36:4.5)$);
		\draw ($(72:\R)$)--($(36:4.5)$);
		\draw[fill=white] ($(36:4.5)$) ellipse [x radius=2.5, y radius=2.5];
		\draw ($(36:4.5)+(0,3)$) node[fill=white, rectangle] { $P_{B,\{b_1, b_2\}}$};

		\draw ($(144:\R)$)--($(180:3)$);
		\draw ($(216:\R)$)--($(180:3)$);
		\draw[fill=white] ($(180:3)$) ellipse [x radius=1.3, y radius=1.3];
		\draw ($(180:3)+(0,-1.8)$) node[fill=white, rectangle] { $P_{B,\{b_3, b_4\}}$};

		\node[label = above: $z_{1}$] () at (-20:3) {};
		\node[label = above: $z_{2}$] () at (90:3) {};
		\node[label = above: $y_{2}$] (c1) at ($(180:3)+(0.4,0)$) {};
		\node[label = above: $y_{1}$] (c2) at ($(180:3)+(-0.4,0)$) {};
		\draw (c1)--(c2);

		\node[label = right: $x_{1,2}$] (a1) at ($(36:4.5)+(0.4,1)$) {};
		\node[label = left: $x_{1,1}$] (a2) at ($(36:4.5)+(-0.4,1)$) {};
		\node[label = right: $x_{2,2}$] (a3) at ($(36:4.5)+(0.4,0.2)$) {};
		\node[label = left: $x_{2,1}$] (a4) at ($(36:4.5)+(-0.4,0.2)$) {};
		\draw ($(36:4.5)+(0,-0.2)$) node[fill=white, rectangle] { $\vdots$};
		\node[label = right: $x_{m,2}$] (a5) at ($(36:4.5)+(0.4,-1)$) {};
		\node[label = left: $x_{m,1}$] (a6) at ($(36:4.5)+(-0.4,-1)$) {};
		\draw (a1)--(a2) (a3)--(a4) (a5)--(a6);
	\end{tikzpicture}
	\caption{
		Two induced subgraphs in the two cases (I) and (II) of Theorem~\ref{thm:key}, where the edges between vertices in different pillars are omitted
	} \label{fig:key}
\end{figure}
\begin{theorem}	\label{thm:key}
	Let $H$ be a graph with smallest Seidel eigenvalue at least $-5$ having a maximum clique $B=\{b_1,\ldots,b_5\}$.
	Assume that the $(5,2)$-pillar $P_{B,\{b_1,b_2\}}$ is isomorphic to $m K_2$ for some non-negative integer $m$,
	and assume one of the following.
	\begin{enumerate}
		\item[\textup{(I)}] The $(5,1)$-pillar $P_{B,\{ b_1 \}}$ contains non-adjacent vertices.
		\item[\textup{(II)}] Both $(5,1)$-pillars $P_{B,\{ b_1\}}$ and $P_{B,\{ b_2\}}$ contain at least one vertex.
	\end{enumerate}
	If the $(5,2)$-pillar $P_{B,\{b_3,b_4\}}$ contains at least one edge, then $m \leq 8$.
\end{theorem}

Our strategy for proving this theorem is to consider an upper bound on $m$ under the condition that the Gram matrix of the vectors $\bar{x}$'s is positive semidefinite, where $x$'s are some vertices of a graph as in Figure~\ref{fig:key}.
In Subsection~\ref{subsec:1}, we provide a lemma which gives the inner products of $\bar{x}$'s, and introduce some notations to represent candidates for the Gram matrix of $\bar{x}$'s.
In Subsection~\ref{subsec:2}, we consider an easier case where there are no $(5,1)$-pillars containing vertices in Figure~\ref{fig:key} precisely.
In Subsection~\ref{subsec:3}, we consider the case where $(5,1)$-pillars containing vertices and $(5,2)$-pillars containing vertices coexist as in Figure~\ref{fig:key}, and complete the proof of Theorem~\ref{thm:key}.

\subsection{The Gram matrix of $\bar{x}$'s}	\label{subsec:1}
In this subsection, we prepare to write down the Gram matrix of the vectors $\bar{x}$'s corresponding to some vertices $x$'s of a graph in Figure~\ref{fig:key}.
King and Tang calculated the inner products of such vectors to obtain upper bounds on the orders of pillars~\cite{KT2019}.
By~\cite[Proposition~3.10]{KT2019}, we have the following lemma.
\begin{lemma}	\label{lem:inner products}
	Let $H$ be a graph with smallest Seidel eigenvalue at least $-5$ having a maximum clique $B=\{b_1,\ldots,b_5\}$.
	For $\{i,j,k,l,m\} = \{1,\ldots,5\}$ the following hold.
	\begin{enumerate}
		\item For a vertex $x$ in the $(5,2)$-pillar $P_{B,\{b_i,b_j\}}$, the orthogonal projection of $\hat{x}$ onto $\langle \hat{b} : b \in B \rangle$ is
		      $
			      \frac{1}{3} \left( \hat{b}_k + \hat{b}_l + \hat{b}_m \right).
		      $
		\item For a vertex $x$ in the $(5,1)$-pillar $P_{B,\{b_i\}}$, the orthogonal projection of $\hat{x}$ onto $\langle \hat{b} : b \in B \rangle$ is
		      $
			      \frac{1}{3} \left( \hat{b}_i + 2\hat{b}_j + 2\hat{b}_k + 2\hat{b}_l + 2\hat{b}_m \right).
		      $
	\end{enumerate}
	In particular, we have the following inner products.
	\begin{enumerate}
		\item[(iii)] For $x,y \in P_{B,\{b_i,b_j\}}$ and $z,w \in P_{B,\{b_i\}}$, the following hold.
		      \begin{align*}
			      \frac{15}{2} \cdot (\bar{x},\bar{y}) = \begin{cases}
				                                             6  & \text{ if } x = y,        \\
				                                             -3 & \text{ if } x \sim y,     \\
				                                             0  & \text{ if } x \not\sim y.
			                                             \end{cases}
			       &  &
			      \frac{15}{2} \cdot (\bar{z},\bar{w}) = \begin{cases}
				                                             4  & \text{ if } z = w,        \\
				                                             -2 & \text{ if } z \not\sim w.
			                                             \end{cases}
		      \end{align*}
		\item[(iv)] For $x \in P_{B,\{b_i,b_j\}}, y \in P_{B,\{b_k,b_l\}}, z \in P_{B,\{b_i\}}, w \in P_{B,\{b_j\}}$, the following hold.
		      \begin{align*}
			       & \frac{15}{2} \cdot (\bar{x},\bar{y}) = \begin{cases}
				                                                -1 & \text{ if } x \sim y,     \\
				                                                2  & \text{ if } x \not\sim y.
			                                                \end{cases}
			       &                                                       &
			      \frac{15}{2} \cdot (\bar{x},\bar{z}) = \begin{cases}
				                                             -3 & \text{ if } x \sim z,     \\
				                                             0  & \text{ if } x \not\sim z.
			                                             \end{cases}
			      \\
			       & \frac{15}{2} \cdot (\bar{y},\bar{z}) = \begin{cases}
				                                                -2 & \text{ if } y \sim z,     \\
				                                                1  & \text{ if } y \not\sim z.
			                                                \end{cases}
			       &                                                       & \frac{15}{2} \cdot (\bar{z},\bar{w}) = \begin{cases}
				                                                                                                        -4 & \text{ if } z \sim w,     \\
				                                                                                                        -1 & \text{ if } z \not\sim w.
			                                                                                                        \end{cases}
		      \end{align*}
	\end{enumerate}
\end{lemma}
In the remainder of this subsection, we introduce notation mainly to represent candidates for the Gram matrix $G$ of
$\bar{z}_1$, $\bar{z}_2$, $\bar{y}_1$, $\bar{y}_2$, $\bar{x}_{1,1}$, $\bar{x}_{1,2}, \ldots, \bar{x}_{m,1}$ and $\bar{x}_{m,2}$ in Figure~\ref{fig:key}.
We first describe the structure of the Gram matrix.
We then introduce notation and, as an example, use it to write down the Gram matrix for the graph in Figure~\ref{fig:exQ}, which is a specific instance of case~\textup{(I)} in Figure~\ref{fig:key}.
Note that in Subsection~\ref{subsec:2} we use the same notation to represent the Gram matrix of different vectors arising in a simpler situation.

\begin{figure}[hbtp]
	\begin{tikzpicture}[baseline=7,xscale=0.6,yscale=0.6,every node/.style={fill=black,circle,minimum size=5 pt,inner sep=0pt}]
		\def\R{1}
		\node[] (b1) at ($(0:\R)$) {};
		\node[] (b2) at ($(72:\R)$) {};
		\node[] (b3) at ($(2*72:\R)$) {};
		\node[] (b4) at ($(3*72:\R)$) {};
		\node[] (b5) at ($(4*72:\R)$) {};

		\def\r{1.6}
		{
			\node[fill=white, rectangle] () at (-12:\r) {$b_1$};
			\node[fill=white, rectangle] () at (89:\r) {$b_2$};
			\node[fill=white, rectangle] () at (149:\r) {$b_3$};
			\node[fill=white, rectangle] () at (231:\r) {$b_4$};
			\node[fill=white, rectangle] () at (303:\r) {$b_5$};
		}

		\foreach \i in {b1,b2,b3,b4,b5}{
				\foreach \j in {b1,b2,b3,b4,b5}{
						\draw (\i)--(\j);
					}
			}
		\coordinate (Z) at (0:4);
		\draw (b1)--($(Z)$);
		\draw[fill=white] ($(Z)$) ellipse [x radius=1, y radius=1];
		\draw ($(Z)+(0,-1.5)$) node[fill=white, rectangle] { $P_{B,\{b_1 \}}$};

		\coordinate (Y) at (140:4);
		\draw (b3)--($(Y)$);
		\draw (b4)--($(Y)$);
		\draw[fill=white] ($(Y)$) ellipse [x radius=1.3, y radius=1.3];
		\draw ($(Y)+(0,+1.8)$) node[fill=none, rectangle] { $P_{B,\{b_3, b_4 \}}$};

		\coordinate (X) at (50:5);
		\draw (b1)--($(X)$);
		\draw (b2)--($(X)$);
		\draw[fill=white] ($(X)$) ellipse [x radius=2.5, y radius=2.5];
		\draw ($(X)+(0,3)$) node[fill=white, rectangle] { $P_{B,\{b_1, b_2 \}}$};

		\node[label = below: $z_{2}$] (z2) at ($(Z)+(0.4,0)$) {};
		\node[label = above: $z_{1}$] (z1) at ($(Z)+(-0.4,0)$) {};
		\node[label = left: $y_{2}$] (y2) at ($(Y)+(0,0.4)$) {};
		\node[label = left: $y_{1}$] (y1) at ($(Y)+(0,-0.4)$) {};
		\draw (y1)--(y2);

		\node[label = right: $x_{1,2}$] (x12) at ($ (X)+(0,1.3-0.3)$) {};
		\node[label = north east: $x_{1,1}$] (x11) at ($(X)+(0,1.3+0.3)$) {};
		\node[label = right: $x_{2,2}$] (x22) at ($(X)+(0,-0.3)$) {};
		\node[label = right: $x_{2,1}$] (x21) at ($(X)+(0,0.3)$) {};
		\node[label = right: $x_{3,2}$] (x32) at ($(X)+(0,-1.3-0.3)$) {};
		\node[label = right: $x_{3,1}$] (x31) at ($(X)+(-0,-1.3+0.3)$) {};
		\draw (x11)--(x12) (x21)--(x22) (x31)--(x32);

		\draw (z1) .. controls (6,1) and (6,6) .. (x11);
		\draw (z1) .. controls (2,-3) and (-4,-3) .. (y1);
		\draw (y1) to[bend left=0] (x22);
		\draw (y1) to[bend left=0] (x32);
		\draw (y2) to[bend left=0] (x11);
		\draw (y2) to[bend left=0] (x12);
		\draw (y2) to[bend left=0] (x21);
		\draw (y2) to[bend left=0] (x31);
	\end{tikzpicture}
	\caption{
		An example of an induced subgraph in the case \textup{(I)} of Theorem~\ref{thm:key}
	}\label{fig:exQ}
\end{figure}

We describe the structure of the Gram matrix $G$, whose entries are calculated using Lemma~\ref{lem:inner products}.
For computational convenience, we work with the integer matrix $Q:=(15/2)G$.
This matrix $Q$ can be written for some matrices $Q_{11}$, $Q_{22}$, $A_1, \ldots, A_m$ as follows.
	{\renewcommand{\arraystretch}{1.5}
		\[
			Q =
			\begin{tabular}{c cccc ccccc}
			& $\bar{z}_1$ & $\bar{z}_2$ & $\bar{y}_1$ & $\bar{y}_2$ & $\bar{x}_{1,1}$ & $\bar{x}_{1,2}$ & $\dots$ & $\bar{x}_{m,1}$ & $\bar{x}_{m,2}$ \\
			\cline{2-10}
		$\bar{z}_1$ & \multicolumn{4}{|c|}{\multirow{4}{*}{$Q_{11}$}} & \multicolumn{2}{|c|}{\multirow{4}{*}{$A_1^\top$}} & \multicolumn{1}{|c|}{\multirow{4}{*}{$\cdots$}} & \multicolumn{2}{|c|}{\multirow{4}{*}{$A_m^\top$}} \\
		$\bar{z}_2$ & \multicolumn{1}{|c}{}& & & \multicolumn{1}{c|}{} & & \multicolumn{1}{c|}{} & \multicolumn{1}{c|}{} & & \multicolumn{1}{c|}{} \\
		$\bar{y}_1$ & \multicolumn{1}{|c}{}& & & \multicolumn{1}{c|}{} & & \multicolumn{1}{c|}{} & \multicolumn{1}{c|}{} & & \multicolumn{1}{c|}{} \\
		$\bar{y}_2$ & \multicolumn{1}{|c}{}& & & \multicolumn{1}{c|}{} & & \multicolumn{1}{c|}{} & \multicolumn{1}{c|}{} & & \multicolumn{1}{c|}{} \\
			\cline{2-10}
		$\bar{x}_{1,1}$ & \multicolumn{4}{|c|}{\multirow{2}{*}{$A_1$}} & \multicolumn{5}{|c|}{\multirow{5}{*}{$Q_{22}$}} \\
		$\bar{x}_{1,2}$ & \multicolumn{1}{|c}{}& & & \multicolumn{1}{c|}{} & & & & & \multicolumn{1}{c|}{} \\
			\cline{2-5}
		$\vdots$ & \multicolumn{4}{|c|}{$\vdots$} & & & & & \multicolumn{1}{c|}{} \\
			\cline{2-5}
		$\bar{x}_{m,1}$ & \multicolumn{4}{|c|}{\multirow{2}{*}{$A_m$}} & & & & & \multicolumn{1}{c|}{} \\
		$\bar{x}_{m,2}$ & \multicolumn{1}{|c}{}& & & \multicolumn{1}{c|}{} & & & & & \multicolumn{1}{c|}{} \\
			\cline{2-10}\\
			\end{tabular}
		\]}
Here, the entries of $Q$ are $15/2$ times the corresponding inner products among vectors associated with labels.
The matrix $Q_{11}$ is determined by the adjacency relations among $z_1$, $z_2$, $y_1$ and $y_2$.
Also, the  matrix $Q_{22}$ equals $(9I_2-3J_2)^{\oplus m}$.
For $i\in\{1,\ldots,m\}$, the matrix $A_i$ is determined by the adjacency relations between $\{x_{i,1},x_{i,2}\}$ and $\{z_1,z_2,y_1,y_2\}$.
We define a function $a$ whose domain is the set of candidate $2\times 4$ matrices for $A_i$.
For each such matrix $A$, the value $a(A)$ is the number of indices $i$ for which $A_i=A$.
Then, since $m$ is the sum of values of $a$, we see that the matrix $Q$ is determined by $Q_{11}$ and $a$.

We now introduce notation to represent the matrix $Q$, which depends only on $Q_{11}$ and $a$.
Let $s,t$ and $r$ be positive integers, and $Z$ and $Y$ be sets of numbers.
Denote by $M_{r,s}(Z)$ the set of $r \times s$-matrices all of whose entries are in $Z$,
and write $M_r(Z)$ for $M_{r,s}(Z)$ if $r=s$.
Denote by $M_{r,s,t}(Z,Y)$ the set of $r \times (s+t)$-matrices obtained by joining a matrix in $M_{r,s}(Z)$ and one in $M_{r,t}(Y)$ horizontally.
For example,
\begin{align*}
	M_{2,1,2}(\{0,1\},\{2,3\}) = \left\{ \begin{bmatrix} i & e & f \\ j & g & h \end{bmatrix} : i,j \in \{0,1\}, e,f,g,h \in \{2,3\} \right\}.
\end{align*}
Let  $M$ be a finite set of $2 \times r$ matrices, and $a : M \to \Z_{\geq 0}$ a function.
Define $Q_{21} = Q_{21}\left(a \right)$ as the matrix obtained by joining all $a(A)$ copies of $A \in M$  vertically.
Let $m$ be the sum of values of the function $a$, and define
\begin{align*}
	Q_{22}=Q_{22}(m) := (9I_2-3J_2)^{\oplus m} =  \begin{bmatrix}
		                                              6 & -3 \\ -3 & 6
	                                              \end{bmatrix}^{\oplus m}.
\end{align*}
In addition, let $Q_{11}$ be an $r \times r$ matrix, and define
\begin{align}
	Q = Q\left(Q_{11} ; a\right) :=
	\begin{bmatrix}
		Q_{11} & Q_{21}^\top \\ Q_{21} & Q_{22}
	\end{bmatrix}.
\end{align}

As an example, we consider the graph in Figure~\ref{fig:exQ}, and write down the Gram matrix of the vectors $\bar{z}_1$, $\bar{z}_2$, $\bar{y}_1$, $\bar{y}_2$, $\bar{x}_{1,1}$, $\bar{x}_{1,2}$ , $\bar{x}_{2,1}$, $\bar{x}_{2,2}$ , $\bar{x}_{3,1}$ and $\bar{x}_{3,2}$.
We let $a : M_{2,2,2}(\{0,-3\},\{2,-1\}) \to \Z_{\geq 0}$ be a function such that
$$a\left( \begin{bmatrix} -3 & 0 & 2 & -1 \\ 0 & 0 & 2 & -1 \end{bmatrix}\right)=1, \qquad a\left( \begin{bmatrix} 0 & 0 & 2 & -1 \\ 0 & 0 & -1 & 2 \end{bmatrix}\right)=2,$$ and $a$ takes $0$ on the other matrices.
Then
\begin{align*}
	Q_{21}\left( a \right)
	=
	\left[
		\begin{array}{cccc}
			-3 & 0 & 2 & -1 \\ 0 & 0 & 2 & -1 \\
			\hline
			0  & 0 & 2 & -1 \\ 0 & 0 & -1 & 2 \\
			\hline
			0  & 0 & 2 & -1 \\ 0 & 0 & -1 & 2
		\end{array}
		\right]
\end{align*}
and
\begin{align}\label{eq:Q}
	Q\left(
	\begin{bmatrix}
		4  & -2 & -2 & 1  \\
		-2 & 4  & 1  & 1  \\
		-2 & 1  & 6  & -3 \\
		1  & 1  & -3 & 6
	\end{bmatrix}
	;a \right)
	=
	\left[
		\begin{array}{cccc|cc|cc|cc}
			4  & -2 & -2 & 1  & -3 & 0  & 0  & 0  & 0  & 0  \\
			-2 & 4  & 1  & 1  & 0  & 0  & 0  & 0  & 0  & 0  \\
			-2 & 1  & 6  & -3 & 2  & 2  & 2  & -1 & 2  & -1 \\
			1  & 1  & -3 & 6  & -1 & -1 & -1 & 2  & -1 & 2  \\
			\hline
			-3 & 0  & 2  & -1 & 6  & -3 & 0  & 0  & 0  & 0  \\
			0  & 0  & 2  & -1 & -3 & 6  & 0  & 0  & 0  & 0  \\
			\hline
			0  & 0  & 2  & -1 & 0  & 0  & 6  & -3 & 0  & 0  \\
			0  & 0  & -1 & 2  & 0  & 0  & -3 & 6  & 0  & 0  \\
			\hline
			0  & 0  & 2  & -1 & 0  & 0  & 0  & 0  & 6  & -3 \\
			0  & 0  & -1 & 2  & 0  & 0  & 0  & 0  & -3 & 6
		\end{array}
		\right]
\end{align}
By Lemma~\ref{lem:inner products}, we see that this matrix multiplied by $2/15$ equals the Gram matrix of $\bar{z}_1$, $\bar{z}_2$, $\bar{y}_1$, $\bar{y}_2$, $\bar{x}_{1,1}$, $\bar{x}_{1,2}$ , $\bar{x}_{2,1}$, $\bar{x}_{2,2}$ , $\bar{x}_{3,1}$, $\bar{x}_{3,2}$ in Figure~\ref{fig:exQ}.

\subsection{An upper bound on $m$ without $(5,1)$-pillars}	\label{subsec:2}
In this subsection, we provide Theorem~\ref{thm:m=9} to precisely consider the case where the $(5,1)$-pillars contain no vertices as shown in Figure~\ref{fig:key2}.
In this theorem, the function $a$ is determined by the adjacency relations between $\{ y_1, y_2\}$ and $\{ x_{1,1}, x_{1,2}, \ldots, x_{m,1}, x_{m,2} \}$ in the graph in Figure~\ref{fig:key2}.
\begin{figure}[hbtp]
	\begin{tikzpicture}[baseline=7,xscale=0.6,yscale=0.6,every node/.style={fill=black,circle,minimum size=5 pt,inner sep=0pt}]
		\def\R{1}
		\node[] (b1) at (0:\R) {};
		\node[] (b2) at (72:\R) {};
		\node[] (b3) at ($(2*72:\R)$) {};
		\node[] (b4) at ($(3*72:\R)$) {};
		\node[] (b5) at ($(4*72:\R)$) {};

		\def\r{1.6}
		{\
			\node[fill=white, rectangle] () at (22:\r) {$b_1$};
			\node[fill=white, rectangle] () at (93:\r) {$b_2$};
			\node[fill=white, rectangle] () at (149:\r) {$b_3$};
			\node[fill=white, rectangle] () at (231:\r) {$b_4$};
			\node[fill=white, rectangle] () at (303:\r) {$b_5$};
		}

		\foreach \i in {b1,b2,b3,b4,b5}{
				\foreach \j in {b1,b2,b3,b4,b5}{
						\draw (\i)--(\j);
					}
			}
		\def\an{16}
		\draw ($(0:\R)$)--($(\an:4.5)$);
		\draw ($(72:\R)$)--($(\an:4.5)$);
		\draw[fill=white] ($(\an:4.5)$) ellipse [x radius=2.5, y radius=2.5];
		\draw ($(\an:4.5)+(0,-3)$) node[fill=white, rectangle] { $P_{B,\{b_1, b_2\}}$};

		\draw ($(144:\R)$)--($(180:3)$);
		\draw ($(216:\R)$)--($(180:3)$);
		\draw[fill=white] ($(180:3)$) ellipse [x radius=1.3, y radius=1.3];
		\draw ($(180:3)+(0,-1.8)$) node[fill=white, rectangle] { $P_{B,\{b_3, b_4\}}$};

		\node[label = above: $y_{2}$] (c1) at ($(180:3)+(0.4,0)$) {};
		\node[label = above: $y_{1}$] (c2) at ($(180:3)+(-0.4,0)$) {};
		\draw (c1)--(c2);

		\node[label = right: $x_{1,2}$] (a1) at ($(\an:4.5)+(0.4,1)$) {};
		\node[label = left: $x_{1,1}$] (a2) at ($(\an:4.5)+(-0.4,1)$) {};
		\node[label = right: $x_{2,2}$] (a3) at ($(\an:4.5)+(0.4,0.2)$) {};
		\node[label = left: $x_{2,1}$] (a4) at ($(\an:4.5)+(-0.4,0.2)$) {};
		\draw ($(\an:4.5)+(0,-0.2)$) node[fill=white, rectangle] { $\vdots$};
		\node[label = right: $x_{m,2}$] (a5) at ($(\an:4.5)+(0.4,-1)$) {};
		\node[label = left: $x_{m,1}$] (a6) at ($(\an:4.5)+(-0.4,-1)$) {};
		\draw (a1)--(a2) (a3)--(a4) (a5)--(a6);
	\end{tikzpicture}
	\caption{The graph considered in Theorem~\ref{thm:m=9}, where the edges between pillars are not shown} \label{fig:key2}
\end{figure}
Although the theorem has essentially been proved in~\cite[Proof of Theorem~5.3]{lemmens1973},
we provide a proof for the convenience of the readers after definitions and a lemma.
For each $2 \times r$-matrix $A$, let $\bar{A}$ be the matrix obtained from $A$ by exchanging the first and second row.
Let
$$
	\cM := \left\{ A, \bar{A} : A \in \left\{
	\begin{bmatrix} 2 & -1 \\ -1& 2 \end{bmatrix},
	\begin{bmatrix} 2 & -1 \\ -1 & -1 \end{bmatrix},
	\begin{bmatrix} -1 & 2 \\ -1 & -1 \end{bmatrix}
	\right\} \right\}.
$$
\begin{lemma}[{\cite[Theorem~2.7.1]{brouwer2011spectra}}]	\label{lem:Schur}
	Let $C$ be a positive definite matrix.
	Then a symmetric matrix
	$
		\begin{bmatrix}
			A & B \\ B^\top & C
		\end{bmatrix}
	$
	is positive semidefinite if and only if $A-B C^{-1} B^\top$ is positive semidefinite.
\end{lemma}

\begin{theorem}	\label{thm:m=9}
	Let $a : M_{2}(\{-1,2\}) \to \Z_{\geq 0}$ be a function.
	Let $m$ be the sum of values of $a$.
	If
	$
		Q\left( 9I_2-3J_2; a \right)
	$
	is positive semidefinite, then $m \leq 9$.
	Furthermore, if equality holds, then $a(A) = 0$ for
	$
		A \not\in \cM.
	$
\end{theorem}
\begin{proof}
	Since $Q\left( 9I_2-3J_2; a \right)$ is positive semidefinite, Lemma~\ref{lem:Schur} implies that
	\begin{align*}
		\Delta :=
		(9I_2-3J_2)
		- \sum_{A \in M_{2}(\{-1,2\})}
		a(A) \cdot
		A^\top
		  (9I_2-3J_2)^{-1}
		A
	\end{align*}
	is positive semidefinite.
	Since $M_2(\{-1,2\})$ consists of $6$ matrices in $\cM$ and the other $10$ matrices explicitly appearing below,
	we have
	\begin{align*}
		0 & \leq
		3\begin{bmatrix} 1 \\ 1 \end{bmatrix}^\top
		\Delta
		\begin{bmatrix} 1 \\ 1 \end{bmatrix}
		+
		\begin{bmatrix} 1 \\ -1 \end{bmatrix}^\top
		\Delta
		\begin{bmatrix} 1 \\ -1 \end{bmatrix}                            \\
		  & = 36
		- 4 \sum_{A \in \cM} a(A)
		-32 a\left( \begin{bmatrix} 2 & 2 \\ 2 & 2 \end{bmatrix} \right) \\
		  & - 16\left(
		a\left( \begin{bmatrix} 2 & 2 \\ 2 & -1 \end{bmatrix} \right)
		+a\left( \begin{bmatrix} 2 & 2 \\ -1 & 2 \end{bmatrix} \right)
		+a\left( \begin{bmatrix} 2 & -1 \\ 2 & 2 \end{bmatrix} \right)
		+a\left( \begin{bmatrix} -1 & 2 \\ 2 & 2 \end{bmatrix} \right)
		\right)                                                          \\
		  & - 8 \left(
		a\left( \begin{bmatrix} 2 & 2 \\ -1 & -1 \end{bmatrix} \right)
		+ a\left( \begin{bmatrix} 2 & -1 \\ 2 & -1 \end{bmatrix} \right)
		+ a\left( \begin{bmatrix} -1 & 2 \\ -1 & 2 \end{bmatrix} \right)
		+ a\left( \begin{bmatrix} -1 & -1 \\ 2 & 2 \end{bmatrix} \right)
		+ a\left( \begin{bmatrix} -1 & -1 \\ -1 & -1 \end{bmatrix} \right)
		\right)                                                          \\
		  & \leq 36 - 4m.
	\end{align*}
	Hence $m \leq 9$.
	Moreover, if $m=9$ then $a(A)$ is equal to $0$ for every $A \not\in \cM$.
	This is the desired condition.
\end{proof}

\subsection{An upper bound on $m$ with $(5,1)$-pillars}		\label{subsec:3}
In this subsection, we consider the case where $(5,1)$-pillars containing vertices and $(5,2)$-pillars containing vertices coexist, and prove Theorem~\ref{thm:key}.
To begin with, we consider the vertices $z_1$, $z_2$, $y_1$ and $y_2$ in Figure~\ref{fig:key}, and the corresponding vectors $\bar{z}_1$, $\bar{z}_2$, $\bar{y}_1$ and $\bar{y}_2$, which are given in Definition~\ref{dfn:pillar} with respect to $B:=\{b_1,b_2,b_3,b_4,b_5\}$ in Figure~\ref{fig:key}.
To give the Gram matrix of these vectors, we prepare some matrices as follows.
Let
\begin{align*}
	B^{(\textup{I})}_{11} := \begin{bmatrix} 4 & -2 \\ -2 & 4 \end{bmatrix}, \quad
	 &  & B^{(\textup{II})}_{11} := \begin{bmatrix} 4 & -1 \\ -1 & 4 \end{bmatrix}, \quad
	 &  & B_{22} := \begin{bmatrix} 6 & -3 \\ -3 & 6 \end{bmatrix}.
\end{align*}
Let
\begin{align*}
	 & B^{(1)}_{21} := \begin{bmatrix} 1 & 1 \\ 1 & 1 \end{bmatrix}, \quad
	 &                                                                       & B^{(2)}_{21} := \begin{bmatrix} -2 & 1 \\ 1 & 1 \end{bmatrix}, \quad
	 &                                                                       & B^{(3)}_{21} := \begin{bmatrix} -2 & -2 \\ 1 & 1 \end{bmatrix}, \quad
	 &                                                                       & B^{(4)}_{21} := \begin{bmatrix} -2 & 1 \\ -2 & 1 \end{bmatrix},        \\
	 & B^{(5)}_{21} := \begin{bmatrix} -2 & 1 \\ 1 & -2 \end{bmatrix}, \quad
	 &                                                                       & B^{(6)}_{21} := \begin{bmatrix} -2 & -2 \\ -2 & 1 \end{bmatrix}, \quad
	 &                                                                       & B^{(7)}_{21} := \begin{bmatrix} -2 & -2 \\ -2 & -2 \end{bmatrix}
\end{align*}
and
\begin{align*}
	B(X,i) :=
	\begin{bmatrix}
		B^{(X)}_{11} & B^{(i) \top}_{21} \\ B^{(i)}_{21} & B_{22}
	\end{bmatrix}
	\quad \text{ for } \ X \in \{\textup{I}, \textup{II}\} \ \text{ and }\  i \in \{1,\ldots,7\}.
\end{align*}
By Lemma~\ref{lem:inner products}, we see that the Gram matrix of $\bar{z}_1$, $\bar{z}_2$, $\bar{y}_1$ and $\bar{y}_2$ is equal to $(2/15) B(X,i)$ for some $ X \in \{\textup{I}, \textup{II}\}$ and $i \in \{1,\ldots,7\}$ by exchanging $z_1$ and $z_2$, or $y_1$ and $y_2$ if necessary.
For example, if $z_1$, $z_2$, $y_1$ and $y_2$ are vertices in Figure~\ref{fig:exQ},
then $X = \textup{I}$ and $i = 2$ as in~\eqref{eq:Q}.

In the remainder of this subsection, for each choice of adjacency relations among $z_1$, $z_2$, $y_1$, $y_2$, or equivalently, $B(X,i)$, we consider an upper bound on $m$ under the condition that the matrix $Q=Q(B(X,i),a)$ is positive semidefinite.
Here, recall that $a$ is determined by the adjacency relations between $\{ z_1, z_2, y_1, y_2\}$ and $\{ x_{1,1}, x_{1,2}, \ldots, x_{m,1}, x_{m,2} \}$.
In the following lemma, we list the cases where the adjacency relations between $\{z_1,z_2\}$ and $\{y_1,y_2\}$ are incompatible with the requirement that $Q$ be positive semidefinite.
\begin{lemma}	\label{lem:1}
	Let $a : M_{2,2,2}(\{-3,0\},\{-1,2\}) \to \Z_{\geq 0}$ be a function.
	If $(X,i) \in \{(\uI,6),(\uI,7),(\uII,7)\}$, then
	$
		Q\left(
		B(X,i) ; a
		\right)
	$
	is not positive semidefinite.
\end{lemma}
\begin{proof}
	By direct calculation, we see that $B(X,i)$ is not positive semidefinite.
	Since $B(X,i)$ is a principal submatrix of $Q(B(X,i),a)$, the matrix $Q(B(X,i),a)$ is not positive semidefinite.
\end{proof}

Next, we treat the remaining cases.
\begin{lemma}	\label{lem:2}
	Let $a : M_{2,2,2}(\{-3,0\},\{-1,2\}) \to \Z_{\geq 0}$ be a function.
	Assume $(X,i) \in \{(\uI,1),(\uI,3),(\uI,4)\} \cup \{(\uII,1),(\uII,3),(\uII,4),(\uII,6)\}.$
	The sum of values of $a$ is at most $6$ if $Q(B(X,i),a)$ is positive semidefinite.
\end{lemma}
\begin{proof}
	By Lemma~\ref{lem:Schur}, we see that
	\begin{align}\label{lem:1:1}
		\Delta :=
		B(X,i)
		- \sum_{A \in M_{2,2,2}(\{-3,0\},\{-1,2\})}
		a(A) \cdot
		A^\top
		  (9I_2-3J_2)^{-1}
		A
	\end{align}
	is positive semidefinite.
	We let
	\begin{align*}
		\bx := \begin{cases}
			       \begin{bmatrix} 1 & 1 & -1 & -1\end{bmatrix}^\top & \text{ if } (X,i) \in \{(\uI,1),(\uII,1)\}, \\
			       \begin{bmatrix} 1 & 1 & 1 & 0 \end{bmatrix}^\top  & \text{ if } (X,i) \in \{(\uI,3),(\uII,3)\}, \\
			       \begin{bmatrix} 1 & 0 & 1 & 1 \end{bmatrix}^\top  & \text{ if } (X,i) \in \{(\uI,4),(\uII,4)\}, \\
			       \begin{bmatrix} 1 & 1 & 1 & 1 \end{bmatrix}^\top  & \text{ if } (X,i) = (\uII,6).
		       \end{cases}
	\end{align*}
	We regard $\bx^\top \Delta \bx$ as a linear polynomial with variables $a(A)$'s.
	The constant term satisfies
	\begin{align*}
		\bx^\top
		B(X,i)
		\bx
		= \begin{cases}
			  2 & \text{ if } (X,i) \in \{(\uI,1),(\uI,3),(\uI,4)\} \cup \{(\uII,4),(\uII,6)\}, \\
			  4 & \text{ if } (X,i) \in \{(\uII,1),(\uII,3)\}.
		  \end{cases}
	\end{align*}
	Also, the coefficients satisfy
	\begin{align*}
		\min \left\{
		\left( A \bx \right)^\top
		                     (9I_2-3J_2)^{-1}
		A
		\bx :
		A \in M_{2,2,2}(\{-3,0\},\{-1,2\})
		\right\} = \frac{2}{3}.
	\end{align*}
	Since $\bx^\top \Delta \bx \geq 0$,
	the sum of values of $a$ is at most $4/(2/3) = 6$.
\end{proof}

\begin{lemma}	\label{lem:3}
	Let $a : M_{2,2,2}(\{-3,0\},\{-1,2\}) \to \Z_{\geq 0}$ be a function.
	Assume $X \in \{\uI,\uII\}$.
	The sum of values of $a$ is at most $8$ if $Q(B(X,5),a)$ is positive semidefinite.
\end{lemma}
\begin{proof}
	Let $m$ be the sum of values of $a$.
	By Theorem~\ref{thm:m=9}, $m \leq 9$ holds.
	By way of contradiction, we assume $m=9$.
	By applying Theorem~\ref{thm:m=9} again,
	$$
		a\left( \begin{bmatrix} A_1 & A_2 \end{bmatrix} \right) = 0
	$$
	holds for every $A_1 \in M_{2}(\{-3,0\})$ and $A_2 \not\in \cM$.
	We define a positive semidefinite matrix $\Delta$ as~\eqref{lem:1:1}.
	Let
	\begin{align*}
		\bx_1 := \begin{bmatrix} 1 & 0 & 1 & -1 \end{bmatrix}^\top, \qquad
		\bx_2 := \begin{bmatrix} 1 & 1 & 1 & 1 \end{bmatrix}^\top, \qquad
		\bx_3 := \begin{bmatrix} 0 & 1 & -1 & 1 \end{bmatrix}^\top.
	\end{align*}
	We regard $ \sum_{i=1}^3 \bx_i^\top \Delta \bx_i$ as a linear polynomial with variables $a(A)$'s.
	The constant term satisfies
	\begin{align*}
		\sum_{i=1}^3 \bx_i^\top
		B(X,5)
		\bx_i
		= \begin{cases}
			  38 & \text{ if } X=\uI  \\
			  40 & \text{ if } X=\uII
		  \end{cases}
	\end{align*}
	Also, the coefficients satisfy
	\begin{align*}
		\min \left\{ \sum_{i=1}^3
		\left( \begin{bmatrix} A_1 & A_2 \end{bmatrix} \bx_i \right)^\top
		                                                             (9I_2-3J_2)^{-1}
		\begin{bmatrix} A_1 & A_2 \end{bmatrix}
		\bx_i :
		A_1 \in M_{2}(\{-3,0\}), A_2 \in \cM
		\right\} = \frac{14}{3}.
	\end{align*}
	Hence $m \leq \lfloor 40/(14/3) \rfloor = 8$.
	This is a contradiction.
	We have $m \leq 8$.
\end{proof}

Let $\cM'_\uI$ be the set of
\begin{align}\label{thm:key:3}
	\begin{bmatrix} 0 & 0 & -1 & -1 \\ 0 & 0 & -1 & 2 \end{bmatrix},
	\begin{bmatrix} 0 & 0 & -1 & -1 \\ 0 & 0 & 2 & -1 \end{bmatrix},
	\begin{bmatrix} 0 & 0 & -1 & 2 \\ 0 & 0 & 2 & -1 \end{bmatrix},
	\begin{bmatrix} -3& 0 & 2 & -1 \\ 0 & 0 & -1 & -1 \end{bmatrix},
	\begin{bmatrix} 0 & -3 & -1 & -1 \\ 0 & 0 & 2 & -1 \end{bmatrix}.
\end{align}
Let $\cM'_\uII$ be the union of $\cM'_\uI$ and the set of
\begin{align}\label{thm:key:4}
	\begin{bmatrix} -3 & 0 & 2 & -1 \\ 0 & 0 & -1 & 2 \end{bmatrix},
	\begin{bmatrix} 0& -3 & -1 & -1 \\ 0 & 0 & -1 & 2 \end{bmatrix},
	\begin{bmatrix} -3& 0 & 2 & -1 \\ 0 &-3 &-1 & -1 \end{bmatrix}.
\end{align}
For $X \in \{\uI,\uII\}$,
we let $$\cM_{X} := \cM'_{X} \cup \{ \bar{A} : A \in \cM'_{X} \}.$$
\begin{lemma}	\label{lem:4}
	Let $a : M_{2,2,2}(\{-3,0\},\{-1,2\}) \to \Z_{\geq 0}$ be a function.
	Let $X \in \{\uI,\uII\}$.
	If $Q(B(X,2),a)$ is positive semidefinite and the sum of values of $a$ equals $9$,
	then $a(A) = 0$ holds for every $A \not\in \cM_X$.
\end{lemma}
\begin{proof}
	Let $m$ be the sum of values of $a$.
	Assume $m=9$.
	We define a positive semidefinite matrix $\Delta$ as~\eqref{lem:1:1}.
	By applying Theorem~\ref{thm:m=9},
	we have
	$
		a(\begin{bmatrix} A_1 & A_2 \end{bmatrix}) = 0
	$
	for any $A_1 \in M_2(\{-3,0\})$ and $A_2 \not\in \cM$.
	We let
	\begin{align*}
		\bx_1 := \begin{bmatrix} 0 &-1 & 1 & 1 \end{bmatrix}^\top, \qquad
		\bx_2 := \begin{bmatrix} 1 & 0 & 1 & 0 \end{bmatrix}^\top, \qquad
		\bx_3 := \begin{bmatrix} 1 & 1 & 0 &-1 \end{bmatrix}^\top.
	\end{align*}
	We regard $\bx_i^\top \Delta \bx_i$ as a linear polynomial with variables $a(A)$'s.
	The constant term satisfies
	$
		\bx_i^\top
		B(X,2)
		\bx_i
		= 6
	$
	for each $(X,i) \in \{(\uI,1),(\uI,2),(\uI,3)\} \cup \{(\uII,1),(\uII,2)\}$.
	Also, the coefficients satisfy
	\begin{align*}
		\min \left\{
		\left( \begin{bmatrix} A_1 & A_2 \end{bmatrix} \bx_i \right)^\top
		                                                             (9I_2-3J_2)^{-1}
		\begin{bmatrix} A_1 & A_2 \end{bmatrix}
		\bx_i :
		A_1 \in M_{2}(\{-3,0\}),  A_2 \in \cM
		\right\} = \frac{2}{3}
	\end{align*}
	for $i \in \{1,2,3\}$.
	Hence for $(X,i) \in \{(\uI,1),(\uI,2),(\uI,3)\} \cup \{(\uII,1),(\uII,2)\}$,
	\begin{align*}
		0 & \leq \bx^\top_i \Delta \bx_i \\
		  & =
		\bx_i^\top
		B(X,2)
		\bx_i
		- \sum_{A \in M_{2,2,2}(\{-3,0\},\{-1,2\})}
		a\left( A \right) \cdot
		\left( A \bx_i \right)^\top
		                       (9I_2-3J_2)^{-1}
		A \bx_i                          \\
		  & =
		\bx_i^\top
		B(X,2)
		\bx_i
		- \sum_{A_1 \in M_2(\{-3,0\}), A_2 \in \cM}
		a\left( \begin{bmatrix} A_1 & A_2 \end{bmatrix} \right) \cdot
		\left( \begin{bmatrix} A_1 & A_2 \end{bmatrix} \bx_i \right)^\top
		                                                             (9I_2-3J_2)^{-1}
		\begin{bmatrix} A_1 & A_2 \end{bmatrix}
		\bx_i                            \\
		  & =
		- \sum_{A_1 \in M_2(\{-3,0\}), A_2 \in \cM}
		a\left( \begin{bmatrix} A_1 & A_2 \end{bmatrix} \right) \cdot
		\left(
		\left( \begin{bmatrix} A_1 & A_2 \end{bmatrix} \bx_i \right)^\top
		                                                             (9I_2-3J_2)^{-1}
		\begin{bmatrix} A_1 & A_2 \end{bmatrix}
		\bx_i
		-\frac{2}{3}
		\right).
	\end{align*}
	We may verify by direct calculation that
	$$
		\left( \begin{bmatrix} A_1 & A_2 \end{bmatrix} \bx_i \right)^\top
		(9I_2-3J_2)^{-1}
		\begin{bmatrix} A_1 & A_2 \end{bmatrix} \bx_i
		> \frac{2}{3},
	$$
	for some $i$ if $\begin{bmatrix} A_1 & A_2 \end{bmatrix} \not\in \cM_X$.
	Hence, $a(A) = 0$ holds for every $A \not\in \cM_X$.
\end{proof}


\begin{proof}[Proof of Theorem~\ref{thm:key}]
	Let $x_{i,1}$ and $x_{i,2}$ be adjacent vertices of $P_{B,\{b_1,b_2\}}$ for $i \in \{1,\ldots,m\}$.
	Let $y_1$ and $y_2$ be adjacent vertices of $P_{B,\{b_3,b_4\}}$.
	In the case of $(\uI)$, we set $X := \uI$, and let $z_1$ and $z_2$ be vertices of $P_{B,\{b_1\}}$.

	In the case of $(\uII)$, we set $X := \uII$.
	Let $z_1$ be a vertex of $P_{B,\{b_1\}}$, and $z_2$ one in $P_{B,\{b_2\}}$.
	Then $z_1$ and $z_2$ are not adjacent.
	Indeed, if $z_1$ and $z_2$ are adjacent, then we have by Lemma~\ref{lem:inner products}
	$
		(\bar{z}_1,\bar{z}_1) = (\bar{z}_2,\bar{z}_2) = 8/15,
	$
	and
	$
		(\bar{z}_1,\bar{z}_2) = -8/15
	$.
	These imply $\bar{z}_1 = - \bar{z}_2$.
	Also Lemma~\ref{lem:inner products} asserts $(\bar{z}_1,\bar{y}_1), (\bar{z}_2,\bar{y}_1) \in \{2/15,-4/15\}$.
	Hence we have a contradiction, and see that $z_1$ and $z_2$ are not adjacent.

	Next we describe the Gram matrix $G$ of $\bar{z}_1,\bar{z}_2,\bar{y}_1,\bar{y}_2,\bar{x}_{1,1},\bar{x}_{1,2},\ldots,\bar{x}_{m,1},\bar{x}_{m,2}$.
	By Lemma~\ref{lem:inner products}, the Gram matrix of $\bar{z}_1$, $\bar{z}_2$, $\bar{y}_1$ and $\bar{y}_2$ multiplied by $15/2$
	is
	\[
		Q_{11} := \begin{bmatrix} B^{(X)}_{11} & B_{21}^\top \\ B_{21} & B_{22} \end{bmatrix}
	\]
	for some $B_{21} \in M_2(\{1,-2\})$.
	By Lemma~\ref{lem:inner products} again, the Gram matrix of $\bar{x}_{1,1},\bar{x}_{1,2},\ldots,\bar{x}_{m,1},\bar{x}_{m,2}$ multiplied by $15/2$ is
	$
		(9I_{2} - 3J_{2})^{\oplus m}.
	$
	By Lemma~\ref{lem:inner products}, for each $i \in \{1,\ldots,m\}$, let
	$$A_i := \frac{15}{2}\begin{bmatrix}
			(\bar{x}_{i,1},\bar{z}_1) & (\bar{x}_{i,1},\bar{z}_2) & (\bar{x}_{i,1},\bar{y}_1) & (\bar{x}_{i,1},\bar{y}_2) \\
			(\bar{x}_{i,2},\bar{z}_1) & (\bar{x}_{i,2},\bar{z}_2) & (\bar{x}_{i,2},\bar{y}_1) & (\bar{x}_{i,2},\bar{y}_2)
		\end{bmatrix} \in M_{2,2,2}(\{-3,0\},\{-1,2\}).
	$$
	We define a function $a : M_{2,2,2}(\{-3,0\},\{-1,2\}) \to \Z_{\geq 0}$ by
	$$
		a(A) := \left| \{ j \in \{1,\ldots,m\} : A_j = A \} \right|
	$$
	for each $A \in M_{2,2,2}(\{-3,0\},\{-1,2\})$.
	Then, we find that
	\[
		Q\left(
		Q_{11};
		a
		\right)
		= \frac{15}{2} G.
	\]

	Let $m$ be the sum of values of $a$.
	By exchanging $y_1$ and $y_2$, or $z_1$ and $z_2$ if necessary,
	we may assume $B_{21} = B_{21}^{(i)}$ for some $i \in \{1,\ldots,7\}$.
	Accordingly, we may write $Q_{11} = B(X,i)$, and hence $Q(Q_{11};a) = Q(B(X,i);a)$.
	By Lemmas~\ref{lem:1}, \ref{lem:2} and~\ref{lem:3}, we have $m \leq 8$ if $i \neq 2$.
	Hence we consider the case of $i=2$.

	We write $C_1,\ldots,C_5$ for the matrices in~\eqref{thm:key:3} in order from left to right.
	In addition, we write $C_6, C_7$ and $C_8$ for the three matrices in~\eqref{thm:key:4} in order from left to right.
	Note that
	$$
		\cM'_{\uI} = \{C_1, C_2, C_3, C_4,C_5\}
		\quad\text{ and }\quad
		\cM'_{\uII} = \cM'_{\uI} \cup \{C_6,C_7,C_8\}.
	$$
	We have
	\begin{align*}
		C_1^\top (9I_2-3J_2)^{-1} C_1 = C_2^\top (9I_2-3J_2)^{-1} C_2 = C_3^\top (9I_2-3J_2)^{-1} C_3
		= \frac{1}{9} \cdot \begin{bmatrix} 0 & 0 & 0 & 0 \\ 0 & 0 & 0 & 0  \\ 0 & 0 & 6 & -3 \\ 0 & 0 & -3 & 6 \end{bmatrix}
	\end{align*}
	and
	\begin{align*}
		 & C_4^\top (9I_2-3J_2)^{-1} C_4
		= \frac{1}{9} \cdot \begin{bmatrix} 18 & 0 & -9 & 9 \\ 0 & 0 & 0 & 0  \\ -9 & 0 & 6 & -3 \\ 9 & 0 & -3 & 6 \end{bmatrix},
		 &                               & C_5^\top (9I_2-3J_2)^{-1} C_5
		= \frac{1}{9} \cdot \begin{bmatrix} 0 & 0 & 0 & 0 \\ 0 & 18 & 0 & 9  \\ 0 & 0 & 6 & -3 \\ 0 & 9 & -3 & 6 \end{bmatrix}, \\
		 & C_6^\top (9I_2-3J_2)^{-1} C_6
		= \frac{1}{9} \cdot \begin{bmatrix} 18 & 0 & -9 & 0 \\ 0 & 0 & 0 & 0  \\ -9 & 0 & 6 & -3 \\ 0 & 0 & -3 & 6 \end{bmatrix},
		 &                               & C_7^\top (9I_2-3J_2)^{-1} C_7
		= \frac{1}{9} \cdot \begin{bmatrix} 0 & 0 & 0 & 0 \\ 0 & 18 & 9 & 0  \\ 0 & 9 & 6 & -3 \\ 0 & 0 & -3 & 6 \end{bmatrix}, \\
		 & C_8^\top (9I_2-3J_2)^{-1} C_8
		= \frac{1}{9} \cdot \begin{bmatrix} 18 & 9 & -9 & 9 \\ 9 & 18 & 0 & 9  \\ -9 & 0 & 6 & -3 \\ 9 & 9 & -3 & 6 \end{bmatrix}.
	\end{align*}

	Theorem~\ref{thm:m=9} asserts $m \leq 9$.
	In order to prove $m \leq 8$ by way of contradiction, we assume $m=9$.
	By Lemma~\ref{lem:4}, we have
	$a(A) = 0$ for $A \not\in \cM_X$.
	Also, we may assume that $a(A) = 0$ for each $A \in \cM_X \setminus \cM_X'$
	by exchanging $x_{i,1}$ and $x_{i,2}$ ($i \in \{1,\ldots, m \}$) if necessary.
	Let $\Delta$ be the matrix defined as~\eqref{lem:1:1}.
	By Lemma~\ref{lem:Schur}, the matrix $\Delta$ is positive semidefinite.
	Below we consider the values of $a(A)$ with $A \in \cM'_{X}$ for $\Delta$ to be positive semidefinite.
	We have
	\begin{align*}
		\Delta
		 & = B(X,2)
		- \sum_{A \in M_{2,2,2}(\{-3,0\},\{-1,2\})}
		a(A) \cdot 	A^\top (9I_2-3J_2)^{-1} A \\
		 & = B(X,2)
		- \sum_{A \in \cM'_X}
		a(A) \cdot 	A^\top (9I_2-3J_2)^{-1} A \\
		 & = B(X,2)
		- \sum_{i=1}^8
		a(C_i) \cdot  C_i^\top (9I_2-3J_2)^{-1} C_i.
	\end{align*}
	Noting that $a(C_1)+\cdots+a(C_8) = m =9$,
	we obtain
	\begin{align*}
		\Delta
		 & =
		\left[
			\begin{array}{c|c}
				B_{11}^{(X)}                                & \begin{matrix} -2 & 1 \\ 1 & 1 \end{matrix} \\
				\hline
				\begin{matrix} -2 & 1 \\ 1 & 1 \end{matrix} & \begin{matrix} 0 & 0 \\ 0 & 0 \end{matrix}
			\end{array}
			\right]
		- a(C_4) \cdot \begin{bmatrix} 2 & 0 & -1 & 1 \\ 0 & 0 & 0 & 0  \\ -1 & 0 & 0 & 0 \\ 1 & 0 & 0 & 0 \end{bmatrix}
		- a(C_5) \cdot \begin{bmatrix} 0 & 0 & 0 & 0 \\ 0 & 2 & 0 & 1  \\ 0 & 0 & 0 & 0 \\ 0 & 1 & 0 & 0 \end{bmatrix} \\
		 & \quad
		- a(C_6) \cdot \begin{bmatrix} 2 & 0 & -1 & 0 \\ 0 & 0 & 0 & 0  \\ -1 & 0 & 0 & 0 \\ 0 & 0 & 0 & 0 \end{bmatrix}
		- a(C_7) \cdot \begin{bmatrix} 0 & 0 & 0 & 0 \\ 0 & 2 & 1 & 0  \\ 0 & 1 & 0 & 0 \\ 0 & 0 & 0 & 0 \end{bmatrix}
		- a(C_8) \cdot \begin{bmatrix} 2 & 1 & -1 & 1 \\ 1 & 2 & 0 & 1  \\ -1 & 0 & 0 & 0 \\ 1 & 1 & 0 & 0 \end{bmatrix}.
	\end{align*}
	Here note that $a(C_6)=a(C_7)=a(C_8)=0$ if $X=\uI$.
	However, if $X=\uI$, then the principal submatrix of $\Delta$ indexed by $\{2,3\}$ has negative determinant.
	This is a contradiction, and $m \leq 8$ holds.

	Next we consider the case of $X=\uII$.
	Then the submatrix of $\Delta$ indexed by $\{1,2\} \times \{3,4\}$ is zero.
	Indeed, for $i \in \{1,2\}$ and $j \in \{3,4\}$, the principal submatrix of $\Delta$ indexed by $\{i,j\}$ is positive semidefinite.
	As $\Delta_{jj} = 0$, we have $-\Delta_{ij}^2 = \Delta_{ii} \Delta_{jj} - \Delta_{ij}^2 \geq 0$.
	Hence, $\Delta_{ij} = 0$ as desired.
	Thus, we have $(a(C_4),a(C_5),a(C_6),a(C_7),a(C_8)) \in \{(1,1,1,1,0),(0,0,1,1,1)\}$.
	Then the principal submatrix indexed by $\{1,2\}$ is
	$
		I_2-J_2
	$
	or
	$
		2I_2-2J_2.
	$
	Since these two matrices are not positive semidefinite, we obtain a contradiction.
	Therefore $m \leq 8$.
\end{proof}

\section{$(5,2)$-pillars and $(5,1)$-pillars} \label{sec:(5,2)-pillar}
In this section, we prove Corollary~\ref{cor:ga} by combining Theorem~\ref{thm:key} with the following lemma.
The lemma is obtained by slightly improving~\cite[Proof of Theorem~5.6]{lemmens1973}.
\begin{lemma} \label{lem:dynkin}
	Let $G$ be a connected graph with largest eigenvalue at most $2$.
	Assume that $G$ is isomorphic to neither $\tilde{A}_{t}$ $(t+1 \not\equiv 0 \pmod 3)$ nor $\tilde{D}_{t}$ $(t+1 \not\equiv 2 \pmod 3)$.
	Then there exist non-negative integers $n$ and $m$ such that $G$ contains an induced subgraph isomorphic to $nK_1 + mK_2$ and
	the order of $G$ is at most
	\begin{align*}
		\frac{4n}{3} + 3m.
	\end{align*}
\end{lemma}
\begin{proof}
	The graph $G$ is isomorphic to one of the graphs in Figure~\ref{fig:CD} (cf.~\cite[Theorem~3.1.3]{brouwer2011spectra}).
	If $G$ is not isomorphic to $D_t$, then the graph obtained from $G$ by removing the white vertices in Figure~\ref{fig:CD} is the desired induced subgraph.
	Hence we consider the case where $G$ is isomorphic to $D_t$ for some $t \geq 4$.
	Here we may assume that the vertices of $G$ are indexed as in Figure~\ref{fig:CD}.
	Let $H$ be the graph obtained from $G$ by removing the following vertices.
	\begin{align*}
		\begin{cases}
			\{1\} \cup \{ i \in \{2,\ldots,t-1\} : i \bmod 3 = 0 \}  & \text{ if } t \bmod 3 = 0, \\
			\{2\} \cup \{ i \in \{3, \ldots,t-1\} : i \bmod 3 = 1 \} & \text{ if } t \bmod 3 = 1, \\
			\{ i \in \{1, \ldots,t-1\} : i \bmod 3 = 2 \}            & \text{ if } t \bmod 3 = 2.
		\end{cases}
	\end{align*}
	Then $H$ is isomorphic to
	\begin{align*}
		\begin{cases}
			qK_2            & \text{ if } t \bmod 3 = 0, \\
			3K_1 + (q-1)K_2 & \text{ if } t \bmod 3 = 1, \\
			2K_1 + qK_2     & \text{ if } t \bmod 3 = 2,
		\end{cases}
	\end{align*}
	where $q := \lfloor t/3 \rfloor$.
	We see that $H$ is the desired induced subgraph.
\end{proof}

\begin{corollary}	\label{cor:ga}
	Let $H$ be a graph with smallest Seidel eigenvalue at least $-5$ having a maximum clique $B=\{b_1,\ldots,b_5\}$.
	Assume one of the following.
	\begin{enumerate}
		\item The $(5,1)$-pillar $P_{B,\{ b_1 \}}$ contains non-adjacent vertices.
		\item Both $(5,1)$-pillars $P_{B,\{ b_1\}}$ and $P_{B,\{ b_2\}}$ contain at least one vertex.
	\end{enumerate}
	If the $(5,2)$-pillar $P_{B,\{b_3,b_4\}}$ contains at least one edge,
	then the $(5,2)$-pillar $P_{B,\{b_1,b_2\}}$ is of order at most $26$.
\end{corollary}
\begin{proof}
	Set $G := P_{B,\{b_1,b_2\}}$.
	By Lemma~\ref{lem:inner products}, the graph $G$ has largest eigenvalue at most $2$.
	Hence, by Theorem~\ref{thm:5.2} and Lemma~\ref{lem:dynkin}, there exists an induced subgraph of $G$ isomorphic to $nK_1 + mK_2$ for some non-negative integers $n$ and $m$ such that
	$|V(G)|\leq 4n/3 + 3m$.
	Also Theorem~\ref{thm:2m+n} asserts $2m+n \leq 18$,
	and Theorem~\ref{thm:key} asserts $m \leq 8$.
	Therefore,
	\begin{align*}
		|V(G)| \leq \frac{4n}{3} + 3m = \frac{4}{3} \cdot (2m+n) + \frac{1}{3} \cdot m \leq 24 + \frac{8}{3} < 27.
	\end{align*}

\end{proof}

\section{A proof of the Lemmens--Seidel conjecture for base size $5$}	\label{sec:proof}
In this section, we prove the main result Theorem~\ref{thm:base=5}.
First, we provide an upper bound on the sum of orders of $(5,1)$-pillars,
which is smaller than the upper bound in \cite[Lemma~D.1]{KT2019}.
\begin{lemma}	\label{lem:(5,1)}
	Let $H$ be a graph with smallest Seidel eigenvalue at least $-5$ having a maximum clique $B$ of size $5$.
	Then the sum of orders of $(5,1)$-pillars with respect to $B$  in $H$ is at most $5$.
\end{lemma}
\begin{proof}
	Let $H'$ be the graph obtained from $H$ by removing all $(5,2)$-pillars with respect to $B$.
	Note that the smallest Seidel eigenvalue of $H'$ is at least $-5$.
	If the tuple of orders of $(5,1)$-pillars is $(4,0,0,0,0)$, $(3,1,0,0,0)$, $(2,2,1,0,0)$ or $(2,1,1,1,0)$ up to permutation,
	then we see by direct calculation that the smallest Seidel eigenvalue of the graph $H'$ is less than $-5$.
	Here, note that there are no edges in any $(5,1)$-pillar.
	Thus the sum of orders of $(5,1)$-pillars is at most $5$.
\end{proof}

\begin{theorem}[{\cite[Theorem~4.6~(2)]{Lin2020}}]	\label{thm:Lin-Yu}
	Let $U$ be a set of $n$ equiangular lines with common angle $\arccos(1/5)$ and base size $5$ in dimension $d$.
	Let $H$ be a graph induced by $U$ with maximum clique $B$ of size $5$.
	If at most one $(5,2)$-pillar with respect to $B$ in $H$ contains a vertex, then
	\begin{align}\label{thm:Lin-Yu:1}
		n \leq \left\lfloor \frac{4d + 36}{3} \right\rfloor.
	\end{align}
\end{theorem}

\begin{proof}[Proof of Theorem~\ref{thm:base=5}]
	Let $U$ be a set of $n$ equiangular lines with common angle $\arccos(1/5)$ and base size $5$ in dimension $d$.
	Fix a graph $H$ induced by $U$ such that $H$ has a maximum clique $B=\{b_1,\ldots,b_5\}$ of size $5$.
	Below we consider pillars with respect to $B$ in $H$.
	If at most one $(5,2)$-pillar contains a vertex, then Theorem~\ref{thm:Lin-Yu} gives~\eqref{thm:Lin-Yu:1}.
	Thus we may assume that at least two $(5,2)$-pillars contain vertices.
	Also, by Lemma~\ref{lem:(5,1)}, the sum of orders of $(5,1)$-pillars is at most $5$.

	We may assume that there is a $(5,2)$-pillar of order at least $2$.
	First, we assume that every $(5,2)$-pillar contains no edges.
	Then by Theorem~\ref{thm:2m+n}, we have
	\begin{align*}
		n \leq 5 + 5 + 9 \cdot 24 + 36 = 262.
	\end{align*}

	Secondly we assume that only one $(5,2)$-pillar contains edges.
	Then by Theorems~\ref{thm:2m+n} and~\ref{thm:|P|}, we have
	\begin{align*}
		n \leq 5 + 5 + 9 \cdot 18 + 54 = 226.
	\end{align*}

	Thirdly we assume that at least two $(5,2)$-pillars contain edges, and that at least one $(5,2)$-pillar contains no edges.
	Then by Theorem~\ref{thm:|P|}, we have
	\begin{align*}
		n \leq 5 + 5 + 9 \cdot 27 + 18 = 271.
	\end{align*}

	Below we assume that every $(5,2)$-pillar contains at least one edge.
	We consider the case where a $(5,1)$-pillar is of order at least $2$.
	Without loss of generality, we may assume that $P_{B,\{b_1\}}$ is of order at least $2$.
	Since the base size of $U$ is $5$, we see that the $(5,1)$-pillars have no edges.
	Hence Corollary~\ref{cor:ga} implies that $P_{B,\{b_1,b_i\}}$ ($i=2, 3, 4, 5$) are of order at most $26$.
	In addition, Theorem~\ref{thm:|P|} implies that the other $(5,2)$-pillars are of order at most $27$.
	Hence
	\begin{align*}
		n \leq 5 + 5 + 6 \cdot 27 + 4 \cdot 26 = 276.
	\end{align*}
	Next we consider the other case, where every $(5,1)$-pillar is of order at most $1$.
	Let $k$ be the number of $(5,1)$-pillars of order $1$.
	Without loss of generality, we may assume that $P_{B,\{b_i\}}$ ($i=1,\ldots,k$) is of order $1$.
	If $k=1$, then
	\begin{align*}
		n \leq 5 + 1 + 10 \cdot 27 = 276.
	\end{align*}
	Otherwise by Corollary~\ref{cor:ga}, $(5,2)$-pillars $P_{B,\{b_i,b_j\}}$ ($1 \leq i < j \leq k$) are of order at most $26$.
	Then we have
	\begin{align*}
		n
		\leq 5 + k + \left( 10 - \binom{k}{2} \right) \cdot 27 + \binom{k}{2} \cdot 26
		= 275 + k - \binom{k}{2}
		\leq 276.
	\end{align*}
	This ends the proof.
\end{proof}

\section{Some properties of sets of $57$ equiangular lines with common angle $\arccos(1/5)$ in dimension $18$ found by Greaves et~al.~\cite{Greaves2021}}	\label{sec:ques}
In this section, we answer Questions~\ref{ques:Gary} and~\ref{ques:Cao} in the negative with the aid of a computer.
For each $i \in \{1,\ldots,4\}$, write $F_i$ for the $18 \times 57$ matrix in~\cite[Figures~1--4]{Greaves2021}.
Let $S_i :=  F_i^\top F_i /2-5I$.
Let $L_i$ be the lattice generated by the $57$ columns of $F_i/\sqrt{2}$.
Let $\bff_i$ be the $i$-th column of $F_1/\sqrt{2}$, and
write $L_G := L_1$.

\begin{proposition}	\label{prop:Gary}
	The four lattices $L_1, L_2, L_3$ and $L_4$ are pairwise isometric,
	and their minimum norms are at most $4$.
	In particular, the four sets of $57$ equiangular lines with common angle $\arccos(1/5)$ in dimension $18$ induced by $S_1, S_2, S_3$ and $S_4$ are not contained in the set of $276$ equiangular lines with common angle $\arccos(1/5)$ in dimension $23$.
\end{proposition}
\begin{proof}
	First, we can verify that $L_1, L_2, L_3$ and $L_4$ are pairwise isometric by software such as Magma~\cite{Magma}.
	Next the vector
	\begin{align*}
		\setcounter{MaxMatrixCols}{20}
		\begin{bmatrix}
			0 & 1 & 0 & 0 & 0 & 0 & 0 & 0 & 1 & -1 & -1 & -1 & 1 & 0 & -1 & 0 & 0 & 1
		\end{bmatrix}^\top/ \sqrt{2}
	\end{align*}
	has norm $4$,
	and is represented as
	$$ \bff_{44}-\bff_{48}-\bff_{49}+\bff_{51}-\bff_{52}+\bff_{53}.$$
	This means that the minimum norm of $L_{G}$ is at most $4$.

	Let $S_{W}$ be the Seidel matrix with smallest eigenvalue $-5$ corresponding to the set of $276$ equiangular lines with common angle $\arccos(1/5)$ in dimension $23$,
	and let $L_{W}$ be the lattice generated by $276$ vectors whose Gram matrix is $5I+ S_{W}$.

	If the set of equiangular lines corresponding to $S_i$ is contained in the set of $276$ equiangular lines in dimension $23$ for some $i \in \{1,2,3,4\}$, then $L_G$ is a sublattice of $L_{W}$ up to isometry.
	However, we can verify that the minimum norm of $L_{W}$ equals $5$ by a computer.
	Hence $L_G$ is not a sublattice of $L_{W}$ up to isometry.
	Therefore,  the four sets of equiangular lines corresponding to the Seidel matrices $S_1,S_2,S_3$ and $S_4$ are not contained in the set of $276$ equiangular lines in dimension $23$.
\end{proof}

Recall that the gap in~\cite[Proof of Theorem~4.6~(1)]{Lin2020} is in claiming that a set of equiangular lines with common angle $\arccos(1/5)$, base size $6$ and at least two pillars containing edges must lie in a unique set of $276$ equiangular lines in dimension $23$.
The following together with Proposition~\ref{prop:Gary} implies that the four sets of equiangular lines induced by $S_1, S_2, S_3$ and $S_4$ are counterexamples to their claim.
\begin{proposition}	\label{prop:counter}
	The sets of $57$ equiangular lines with common angle $\arccos(1/5)$ in dimension $18$ induced by $S_1, S_2, S_3$ and $S_4$ have base size $6$ and at least two pillars with edges.
\end{proposition}
\begin{proof}
	Let $G$ be the graph induced by the Seidel matrix $S_1$ with vertex set $V(G)=\{1,\ldots,57\}$.
	Then we can easily check that
	$B := \{9,13,16,17,18,28\}$ is a maximum clique,
	and edges $\{1,54\}$ and $\{ 5, 8 \}$ are contained in two distinct pillars with respect to $B$, respectively.
	Similarly, we may find a desired clique and edges for each of $S_2$, $S_3$ and $S_4$.
\end{proof}

Finally we answer Question~\ref{ques:Cao} in the negative as follows.
\begin{proposition}	\label{prop:strongly}
	The sets of $57$ equiangular lines with common angle $\arccos(1/5)$ in dimension $18$ induced by $S_1, S_2, S_3$ and $S_4$ are strongly maximal.
\end{proposition}
\begin{proof}
	Recall that $L_G$ is generated by $\bff_1,\ldots,\bff_{57}$,
	and $5I+S_1$ equals the Gram matrix of the vectors $\bff_1,\ldots,\bff_{57}$.
	We see that the set is not strongly maximal
	if and only if there is a non-zero vector $\bu \in L^*_G := \{ \bv \in \Q L_G : (\bv,\bw) \in \Z \text{ for every } \bw \in L_G \}$ of norm at most $5$ such that $(\bu,\bff_i) \in \{1,-1\}$ for every $i \in \{1,\ldots,57\}$.
	With a computer, we can verify that such a vector does not exist.
	Hence we see that the set of equiangular lines corresponding to $S_1$ is strongly maximal.
	Similarly, we may obtain the desired result for each of $S_2$, $S_3$ and $S_4$.
\end{proof}

\section*{Declaration of competing interest}
The author declares that there is no conflict of interest in this paper.

\section*{Acknowledgements}
\indent
I am grateful to Akihiro Munemasa for his helpful comments.
This work was supported by JSPS KAKENHI Grant Number JP21J14427.

\bibliographystyle{plain}
\bibliography{references.bib}

\end{document}